\documentclass[3p, times, sort&compress]{elsarticle}
 \pdfoutput=1
 
\usepackage{import}

\usepackage[utf8]{inputenc}

\usepackage[T1]{fontenc}

\usepackage{graphicx}
\usepackage{color}

\usepackage{newtxtext}
\usepackage{amsthm}
\usepackage[slantedGreek]{newtxmath}
\usepackage[OMLmathsfit]{isomath}
\DeclareMathAlphabet{\mathbfsf}{\encodingdefault}{\sfdefault}{bx}{n}
\usepackage{bm}

\usepackage{envmath}
\usepackage{mathtools}  
\usepackage{siunitx}
\usepackage{amssymb}

\usepackage{subcaption}
\usepackage{booktabs}
\usepackage{footmisc}  
\usepackage{tabularx}
\usepackage{url}

\theoremstyle{definition}

\theoremstyle{plain}

\newtheorem{lem}{Lemma}
\newtheorem{cor}{Corollary}
\newtheorem{prop}{Proposition}

\theoremstyle{remark}
\newtheorem*{remark}{Remark}

\usepackage{lineno}
\modulolinenumbers[5]
\usepackage{todonotes}
\usepackage{umoline}

\usepackage{pgfplots}
\usepackage{pgfplotstable}
\pgfplotsset{compat=newest}
\pgfplotsset{plot coordinates/math parser=false}
\newlength\figureheight
\newlength\figurewidth
\pgfplotsset{every axis plot/.append style={line width=1.5pt},
    legend style={font=\footnotesize, 
        text height=1.0ex,
        draw=black,
        fill=white,
        legend cell align=left}}

\usepackage{acro}

\DeclareAcronym{ACA}
{
    short = ACA ,
    long = adaptive cross approximation
}

\DeclareAcronym{EFIE}
{
    short =  EFIE ,
    long = electric field integral equation
}

\DeclareAcronym{CFIE}
{
    short =  CFIE ,
    long = combined field integral equation
}

\DeclareAcronym{SPD}
{
    short =  SPD ,
    long = {symmetric, positive definite}
}

\DeclareAcronym{SPSD}
{
    short =  SPD ,
    long = {symmetric, positive semi-definite}
}

\DeclareAcronym{PEC}
{
    short =  PEC ,
    long = perfectly electrically conducting
}

\DeclareAcronym{RWG}
{
    short = RWG ,
    long = Rao-Wilton-Glisson
} 

\DeclareAcronym{BC}
{
    short = BC ,
    long = Buffa-Christiansen
}

\DeclareAcronym{SVD}
{
    short = SVD ,
    long = singular value decomposition
}

\DeclareAcronym{CG}
{
    short = CG ,
    long = conjugate gradient
} 

\DeclareAcronym{PCG}
{
    short = PCG ,
    long = preconditioned conjugate gradient
} 

\DeclareAcronym{CGS}
{
    short = CGS ,
    long = conjugate gradient squared
}

\DeclareAcronym{CMP}
{
    short = CMP ,
    long = Calderón multiplicative preconditioner
} 

\DeclareAcronym{RFCMP}
{
    short = RF-CMP ,
    long = refinement-free Calderón multiplicative preconditioner
} 

\DeclareAcronym{HPD}
{
    short = HPD ,
    long = {Hermitian, positive definite}
} 

\DeclareAcronym{MLFMM}
{
    short = MLFMM ,
    long = {multilevel fast multipole method}
} 
\usepackage{hyperref}
\usepackage[english]{cleveref}

\Crefname{defn}{definition}{definitions}
\Crefname{defn}{Definition}{Definitions}

\Crefname{asm}{assumption}{assumptions}
\Crefname{asm}{Assumption}{Assumptions}

\crefname{lem}{lemma}{lemmas} 
\Crefname{lem}{Lemma}{Lemmas}

\crefname{prop}{proposition}{propositions} 
\Crefname{prop}{Proposition}{Propositions}

\crefname{thm}{theorem}{theorms} 
\Crefname{thm}{Theorem}{Theorms}

\crefname{cor}{corollary}{corollaries}
\Crefname{cor}{Corollary}{Corollaries}
\newcounter{subequation}
\newlength\mtabskip\mtabskip=-1.25cm

\newcommand\eqnCnt[1][]{ %
    \refstepcounter{subequation} %
    \begin{align}#1\end{align} %
    \addtocounter{equation}{-1}}
\def\mtabLong{long}
\makeatletter
\newenvironment{mtabular}[2][\empty]{ %
    \def\@xarraycr{ %
        \stepcounter{equation} %
        \setcounter{subequation}{0} %
        \@ifnextchar[\@argarraycr{\@argarraycr[\mtabskip]}}
    \let\theoldequation\theequation %
    \renewcommand\theequation{\theoldequation.\alph{subequation}}
    \edef\mtabOption{#1}
    \setcounter{subequation}{0} %
    \tabcolsep=0pt
    \ifx\mtabOption\mtabLong\longtable{#2}\else\tabular{#2}\fi %
}{ %
\ifx\mtabOption\mtabLong\endlongtable\else\endtabular\fi %
\let\theequation\theoldequation %
\stepcounter{equation}}
\makeatother 
\newcommand{\mr}{\mathrm}

\newcommand{\mc}{\mathcal}
\newcommand{\wt}[1]{{\widetilde{#1}}}

\newcommand{\veg}[1]{\bm{#1}}     
\newcommand{\mat}[1]{\mathsfbfit{#1}} 
\renewcommand{\vec}[1]{\mathsfbfit{#1}} 
\newcommand{\op}[1]{\mathcal{#1}} 
\newcommand{\vecop}[1]{\bm{\mathcal{#1}}} 
\newcommand{\mel}[1]{\mathsfit{#1}} 

\newcommand{\matel}[1]{\begin{bmatrix} #1 \end{bmatrix}}
\newcommand{\vecel}[1]{\begin{bmatrix} #1 \end{bmatrix}}

\newcommand{\matI}{\mathbfsf{I}}
\newcommand{\matO}{\mathbfsf{0}}
\newcommand{\vecO}{\matO}
\newcommand{\vecI}{\mathbfsf{1}}

\newcommand{\n}{\hat{\bm{n}}}

\newcommand{\dd}{\mathrm{d}}  

\newcommand{\ii}{\mathrm{i}}   
\newcommand{\im}{\mathrm{i}}  

\newcommand{\e}{\mathrm{e}}

\newcommand{\C}{\mathbb{C}}

\newcommand{\N}{\mathbb{N}}

\newcommand{\R}{\mathbb{R}}

\DeclarePairedDelimiter{\abs}{\lvert}{\rvert}
\DeclarePairedDelimiter{\norm}{\lVert}{\rVert}

\DeclareMathOperator{\cond}{cond}

\DeclareMathOperator{\Null}{null}

\DeclareMathOperator{\Span}{span}

\newcommand{\T}{\mr{T}}

\newcommand\restr[2]{{
        \left.\kern-\nulldelimiterspace 
        #1 
        \vphantom{|} 
        \right|_{#2} 
}}

\newcommand\rst[3]{{
        \left.\kern-\nulldelimiterspace 
        #1 
        \vphantom{|} 
        \right|_{#2}^{#3} 
}}

\newcolumntype {n}{c}
\newcolumntype {N}{>{\small}c}
\newcolumntype {L}{>{\small}l}
\newcolumntype {F}{>{\footnotesize}c}
\newcolumntype {v}[1]{>{\raggedright \hspace {0pt}} p {#1}}
\newcolumntype {V}[1]{>{\small \raggedright \hspace {0pt}} p {#1}}
\newcolumntype{d}[1]{>{\DC@{.}{.}{#1}}c<{\DC@end}}

\newcolumntype{R}[1]{%
    >{\begin{turn}{90}\begin{minipage}{#1}\small\raggedright\hspace{0pt}}l%
            <{\end{minipage}\end{turn}}%
}

\newcommand{\formatedtime}[3]{#1:#2:#3}

\newcommand{\TA}{\vecop T_{\kern-4pt\mr{A}}}
\newcommand{\TPhi}{\vecop T_{\kern-4pt\Phiup}}

\newcommand{\matTA}{\mat T_\mr{A}}
\newcommand{\matTPhi}{\mat T_\Phiup}

\newcommand{\matL}{\mat{\Lambda}}   
\newcommand{\matS}{\mat{\Sigma}}

\newcommand{\Pl}{\mat P_{\upLambda}}
\newcommand{\Ps}{\mat P_{\upSigma}}

\newcommand{\Ph}{\mat P_{\mr H}}
\newcommand{\Plh}{\mat P_{\upLambda   \mr H}}

\newcommand{\wveg}[1]{\widetilde{\veg #1}}
\newcommand{\omat}[1]{\mathring{\mat #1}}
\newcommand{\hmat}[1]{\hat{\mat #1}}
\newcommand{\wmat}[1]{\wt{\mat #1}}
\newcommand{\htmat}[1]{\hat{\widetilde{\mat #1}}}
\newcommand{\cmat}[1]{\check{\mat #1}}
\newcommand{\cwmat}[1]{\check{\widetilde{\mat #1}}}
\newcommand{\whmat}[1]{\widehat{\mat #1}}

\renewcommand{\mel}[1]{\mathbb{#1}}

\newcommand{\bmatL}{\breve{\matL}}
\newcommand{\Po}{\mat P_{\mr o}}
\newcommand{\Pm}{\mat P_{\mr m}}
\newcommand{\PmL}{\mat P_{\mr m \upLambda}}
\newcommand{\PmS}{\mat P_{\mr m \upSigma}}
\newcommand{\bPmS}{{\mat P}_{\mr m \upSigma}}
\newcommand{\bPm}{{\mat P}_{\mr m}}
\newcommand{\bPo}{\Po}
\newcommand{\Pgs}{\mat P_{\mr g \upSigma}}
\newcommand{\Psh}{\mat P_{\upSigma  \mr H}}

\newcommand{\oL}{\vecI_\upLambda}
\newcommand{\oS}{\vecI_\upSigma}

\newcommand{\ooL}{\vecO_\upLambda}

\newcommand{\Laplace}{\Deltaup}

\renewcommand{\mathit}{}
\usepackage{changes}

\begin{document}	
   \begin{frontmatter}

    \title{On a Refinement-Free Calderón Multiplicative Preconditioner for the Electric Field Integral Equation}
    
    \author[hft,polito]{S.~B.~Adrian}
    \ead{simon.adrian@tum.de}
    \author[telecom,polito]{F.~P.~Andriulli}
    \ead{francesco.andriulli@polito.it}
    \author[hft]{T.~F.~Eibert}
    \ead{eibert@tum.de}
    
    \address[hft]{Technical University of Munich, Arcisstr. 21, 80333, Munich, Germany}
    \address[telecom]{École Nationale Supérieure Mines-Télécom Atlantique Bretagne Pays de la Loire, 29238, Brest, France}
    \address[polito]{Politecnico di Torino, 10129, Turin, Italy}
    
    \begin{abstract}
        We present a Calderón preconditioner for the \ac{EFIE}, which does not require a barycentric refinement of the mesh and which yields a \ac{HPD} system matrix allowing for the usage of the \ac{CG} solver. The resulting discrete equation system is immune to the low-frequency and the dense-discretization breakdown and, in contrast to existing Calderón preconditioners, no second discretization of the \ac{EFIE} operator with \ac{BC} functions is necessary. This preconditioner is obtained by leveraging on spectral equivalences between (scalar) integral operators, namely the single layer and the hypersingular operator known from electrostatics, on the one hand, and the Laplace-Beltrami operator on the other hand.
        Since our approach incorporates Helmholtz projectors, there is no search for global loops necessary and thus our method remains stable on multiply connected geometries.
        The numerical results demonstrate the effectiveness of this approach for both canonical and realistic (multi-scale) problems.
    \end{abstract}
    
\end{frontmatter}

\acresetall %

\section{Introduction}
The \ac{EFIE}, which is used for solving electromagnetic scattering and radiation problems, results in an ill-conditioned linear system of equations when discretized with \ac{RWG} basis functions---the common choice in standard codes. The ill-conditioning stems from two issues: the low-frequency breakdown, which is due to the different scaling of the vector and the scalar potential in frequency, and the dense-discretization breakdown, which is due to the fact that the vector and the scalar potential are pseudo-differential operators of negative and positive order~\cite{nedelec_acoustic_2001}, a property that leads in general to ill-conditioned system matrices. Altogether, the condition number of the \ac{EFIE} system matrix grows as $(kh)^{-2}$, where $k$ is the wavenumber and $h$ is the average edge length of the mesh, leading to slowly or non-converging iterative solvers \cite{saad_iterative_2003}.

The low-frequency breakdown has been overcome in the past by using explicit quasi-Helmholtz decompositions such as the loop-star or the loop-tree decomposition \cite{wilton_improving_1981,wu_study_1995,burton_study_1995,vecchi_loopstar_1999, miano_surface_2005}. While these decompositions cure the low-frequency breakdown, the dense-discretization breakdown persists and is even worsened to a $1/h^3$ scaling of the condition number in the case of the loop-star decomposition \cite{andriulli_loopstar_2012}. Some other low-frequency stable methods have been proposed in the past such as the augmented \ac{EFIE}, the rearranged loop-tree decomposition, or the augmented EFIE with normally constrained
magnetic field and static charge extraction \cite{zhao_integral_2000,qian_augmented_2008, cheng_augmented_2015}, yet all of them suffer from $h$-ill-conditioning.

Different strategies have been presented to overcome the dense-discretization breakdown. A first class of techniques relies on algebraic strategies such as the incomplete LU factorization \cite{sertel_incomplete_2000, carpentieri_symmetric_2012}, sparse approximate inverse  \cite{carpentieri_combining_2005, pan_sparse_2014}, or near-range preconditioners \cite{wiedenmann_effect_2013}. While they improve the conditioning, the condition number still grows with decreasing $h$.

In recent years more elaborate explicit quasi-Helmholtz decompositions have been presented, the so-called hierarchical basis preconditioners for both structured \cite{vipiana_multiresolution_2005, andriulli_multiresolution_2007, chen_multiresolution_2009, andriulli_solving_2010} and unstructured meshes \cite{andriulli_hierarchical_2008,adrian_hierarchical_2017} (in the mathematical community, the hierarchical basis preconditioners are typically termed multilevel or prewavelet preconditioners). The best of these methods can yield a condition number that only grows logarithmically in $1/h$ \cite{andriulli_solving_2010, adrian_hierarchical_2017} meaning that there is still space left for improvement.

In fact, Calderón identity-based preconditioners yield a condition number that has an upper bound independent from $h$ \cite{christiansen_preconditioner_2002, buffa_dual_2007,andriulli_multiplicative_2008,contopanagos_wellconditioned_2002,adams_physical_2004}. In the static limit, however, the Calderón strategies stop working due to numerical cancellation in both the right-hand side excitation vector and the unknown current, since solenoidal and non-solenoidal components scale differently in $k$ \cite{zhang_magnetic_2003}. Explicit quasi-Helmholtz decompositions do not suffer from this cancellation since the solenoidal and non-solenoidal components are stored separately. To make Calderón preconditioners stable in the static limit, one could combine the \ac{CMP} with an explicit quasi-Helmholtz decomposition.

Such an approach has a severe downside: if the geometry is multiply connected, then the quasi-harmonic global loop functions have to be added to the basis of the decomposition~\cite{vecchi_loopstar_1999}. In contrast to loop, star, or tree functions, (or any of the hierarchical bases mentioned here), the construction of the global loops becomes costly if the genus $g$ is proportional to the number of unknowns $N$ resulting in the overall complexity $\mc O(N^2 \log(N))$, where $N$ is the number of unknowns (see, for example, the discussion in \cite{adrian_hierarchical_2014a}). In order to avoid the construction of the global loops, a modified \ac{CMP} has been presented which leverages on an implicit quasi-Helmholtz decomposition based on projectors \cite{andriulli_wellconditioned_2013a}. These projectors require the application of the inverse primal (i.e.,  cell-based) and the inverse dual (i.e., vertex-based) graph Laplacian, a task for which blackbox-like preconditioners such as algebraic multigrid methods can be used for rapidly obtaining the inverse. 

In order to avoid the inversion of graph Laplacians, one could think that an alternative might be the Calderón preconditioner combined with an explicit loop-star quasi-Helmholtz decomposition as described in the penultimate paragraph, at least if no global loops are present. However, the inverse Gram matrices appearing in such a scheme are all spectrally equivalent to discretized Laplace-Beltrami operators. Different from the graph Laplacians, these Gram matrices are not symmetric since the loop-star basis is applied to a mixed Gram matrix, that is, \ac{BC} functions are used as expansion and rotated \ac{RWG} functions are used as testing functions \cite{andriulli_loopstar_2012, andriulli_wellconditioned_2013a}. In general, this complication makes it more challenging to stably invert these Gram matrices compared with the graph Laplacians of the Helmholtz projectors since, for example, many algebraic multigrid preconditioners require symmetric matrices.

The work in \cite{andriulli_multiplicative_2008,andriulli_wellconditioned_2013a} demonstrated a Calderón scheme that can be relatively easily integrated in existing codes. Instead of discretizing the operator on the standard mesh with \ac{RWG} and \ac{BC} functions, only a single discretization with \ac{RWG} functions on the barycentrically refined mesh is necessary. The disadvantage of this approach is that the memory consumption as well as the costs for a single-matrix vector product are increased by a factor of six. %

In this work, we propose a \ac{RFCMP} for the \ac{EFIE}. In contrast to existing Calderón preconditioners, no \ac{BC} functions are employed, so that a standard discretization of the \ac{EFIE} with \ac{RWG} functions can be used. What is more, we get a system matrix, which is \ac{HPD}. We obtain this result by leveraging on spectral equivalences between the single layer and the hypersingular operator known from electrostatics on the one hand and the Laplace-Beltrami operator on the other hand. Similar to \cite{andriulli_wellconditioned_2013a}, graph Laplacians need to be inverted. Since the new system matrix is \ac{HPD}, we are allowed to employ the \ac{CG} solver. Different from other Krylov subspace methods, it guarantees convergence and has the least computational overhead. The numerical results corroborate the new formulation. Preliminary results have been presented at conferences \cite{adrian_calderon_2014, adrian_hermitian_2015}.

The paper is structured as follows. \Cref{sec:Background} sets the background and notation; \Cref{sec:TheoreticalApparatus} introduces the new formulation and provides the theoretical apparatus. Numerical results demonstrating the effectiveness of the new approach are shown in  \Cref{sec:NumericalResults}.

\section{Notation and Background}\label{sec:Background}
In the following, we denote quantities residing in $\R^3$, such as the electric field $\veg E$, with an italic, bold, serif font.
For any other vectors and matrices we use an italic, bold, sans-serif font, and we distinguish matrices from vectors by using capital letters for matrices and minuscules for vectors. 
The expression $a \lesssim b$ has to be read as $a \leq C b$, where $C>0$ is a constant independent of the mesh parameter $h$ (i.e., the average edge length).
Furthermore, $a \asymp b$ means that  $a \lesssim b$  and $b \lesssim a$ holds.
Let $f_n \in X_{f}$ and $g_n \in X_{g}$ be functions of the basis function spaces $X_{f}$ and $X_{g}$, respectively; the Gram matrix of these functions is defined and denoted as
\begin{equation}
    \matel{\mat G_{fg}}_{mn} \coloneqq \(f_m, g_n \)_{L^2}\, ,
\end{equation}
where $\(f_n, g_n \)_{L^2} \coloneqq \int_\Gamma f_n (\veg r) g_n  \(\veg r \) \dd S(\veg r)$ is the $L^2$-inner product. Furthermore, quantities related to the dual mesh are denoted with a wide tilde symbol $\wt{\phantom{\lambda}}$.

In the following, several basis functions appear. On the primal mesh, we define the piecewise linear functions
\begin{equation}
    X_{\lambda} \ni  \lambda_n(\veg r) = \begin{cases}
        1 \quad & \text{for~} \veg r \in v_n\, , \\
        0 \quad & \text{for~} \veg r \in v_m \neq v_n \, ,\\
        \text{linear} \quad & \text{elsewhere,} \label{eq:pwlf}
    \end{cases}
\end{equation}
where $v_n \subset \Gamma$ is the $n$th vertex of mesh, the piecewise constant functions
\begin{equation}
    X_{p} \ni p_n(\veg r) = \begin{cases}
        1/{A_n} \quad &\veg r \in c_n\, ,\\
        0 \quad &\text{elsewhere,}
    \end{cases}\label{eq:patchstddef}
\end{equation}
where $c_n$ denotes the domain of the $n$th cell of the mesh, and the \ac{RWG} functions
\begin{equation}
    X_{\veg f} \ni \veg f_n(\veg r) = \begin{cases}
        \dfrac{\veg r - \veg r_n^+}{2 A_{c_n^+}} \quad \text{for~}\veg r\in c_n^+ \, ,\\
        \dfrac{\veg r_n^- - \veg r}{2 A_{c_n^-}} \quad \text{for~}\veg r\in c_n^- \, , \label{eq:RWGdef}
    \end{cases}
\end{equation}
where $c_n^+$ and $c_n^-$ are the domains of the adjacent cells at the $n$th edge, $A_{c_n^+}$ and $A_{c_n^-}$ are the areas of these cells, and $\veg r_n^+$ and $\veg r_n^-$ are the position vectors of the free vertices opposite to the $n$th edge.

\begin{figure}
    \centering
    \begin{tikzpicture}
        \coordinate (L1) at (0,0);
        \coordinate (L2) at (4,-0.5);
        \coordinate (L3) at (1.5,2.5);
        \coordinate (L4) at (4.5,3.2);
        \coordinate (L5) at (-0.75,3.0);
        \coordinate (L6) at (2,-3);
        \coordinate (L7) at (-2,0);
        \coordinate (L8) at (5,-3);
        \coordinate (L9) at (7,0);
        \coordinate (L123) at (barycentric cs:L1=1,L2=1,L3=1);
        \coordinate (L135) at (barycentric cs:L1=1,L3=1,L5=1);
        \coordinate (L157) at (barycentric cs:L1=1,L5=1,L7=1);
        \coordinate (L176) at (barycentric cs:L1=1,L7=1,L6=1);
        \coordinate (L162) at (barycentric cs:L1=1,L6=1,L2=1);
        \coordinate (L268) at (barycentric cs:L2=1,L6=1,L8=1);
        \coordinate (L289) at (barycentric cs:L2=1,L8=1,L9=1);
        \coordinate (L294) at (barycentric cs:L2=1,L9=1,L4=1);
        \coordinate (L243) at (barycentric cs:L2=1,L4=1,L3=1);
        \coordinate (L345) at (barycentric cs:L3=1,L4=1,L5=1);  
      
       \begin{scope}[ultra thick]
             \draw (L1) -- coordinate[midway](L12) (L2) -- coordinate[midway](L23) (L3) -- coordinate[midway](L31) cycle;
             \draw (L2) -- (L3) -- coordinate[midway](L34) (L4) -- coordinate[midway](L42) cycle;
             \draw (L1) -- (L3) -- coordinate[midway](L35) (L5) -- coordinate[midway](L51) cycle;
             \draw (L1) -- (L2) -- coordinate[midway](L26) (L6) -- coordinate[midway](L61) cycle;
             \draw (L1) -- coordinate[midway](L15) (L5) -- coordinate[midway](L57) (L7) -- coordinate[midway](L71) cycle;
             \draw (L7) -- coordinate[midway](L76) (L6) -- coordinate[midway](L68) (L8) -- coordinate[midway](L82) (L2);
             \draw (L8) -- coordinate[midway](L89) (L9) -- coordinate[midway](L92) (L2);
             \draw (L9) -- coordinate[midway](L94) (L4);
             \draw (L4) -- coordinate[midway](L45) (L5);
        \end{scope}
        \begin{scope}
        \draw (L12) -- (L3);
        \draw (L23) -- (L1);
        \draw (L31) -- (L2);
        \draw (L42) -- (L3);
        \draw (L34) -- (L2);
        \draw (L23) -- (L4);
        \draw (L35) -- (L1);
        \draw (L51) -- (L3);
        \draw (L31) -- (L5);
        \draw (L1) -- (L26);
        \draw (L2) -- (L61);
        \draw (L6) -- (L12);
        \draw (L3) -- (L45);
        \draw (L35) -- (L4);
        \draw (L34) -- (L5);
        \draw (L1) -- (L76);
        \draw (L71) -- (L6);
        \draw (L7) -- (L61);
        \draw (L1) -- (L57);
        \draw (L5) -- (L71);
        \draw (L7) -- (L15);
        \draw (L2) -- (L68);
        \draw (L6) -- (L82);
        \draw (L8) -- (L26);
        \draw (L2) -- (L89);
        \draw (L8) -- (L92);
        \draw (L9) -- (L82);
        \draw (L2) -- (L94);
        \draw (L9) -- (L42);
        \draw (L4) -- (L92);
        \filldraw [opacity=0.5, gray] (L135) -- (L51) -- (L157) -- (L71) -- (L176) -- (L61) -- (L162) -- (L26) -- (L268) -- (L268) -- (L82) -- (L289) -- (L92) -- (L294) -- (L42) -- (L243) -- (L34) -- (L345) -- (L35) -- cycle;
        \node[draw, fill=white] at (L123) {1};
        \node[draw, fill=white] at (L12) {1/2};
        \node[draw, fill=white] at (L23) {1/2};
        \node[draw, fill=white] at (L31) {1/2};
        \node[draw, fill=white] at (L1) {1/5};
        \node[draw, fill=white] at (L2) {1/6};
        \node[draw, fill=white] at (L3) {1/4};
        \end{scope}
        \end{tikzpicture}
    \caption{Dual piecewise linear function $\wt \lambda_n$: the gray area denotes the support.}
    \label{fig:DPLF}
\end{figure}
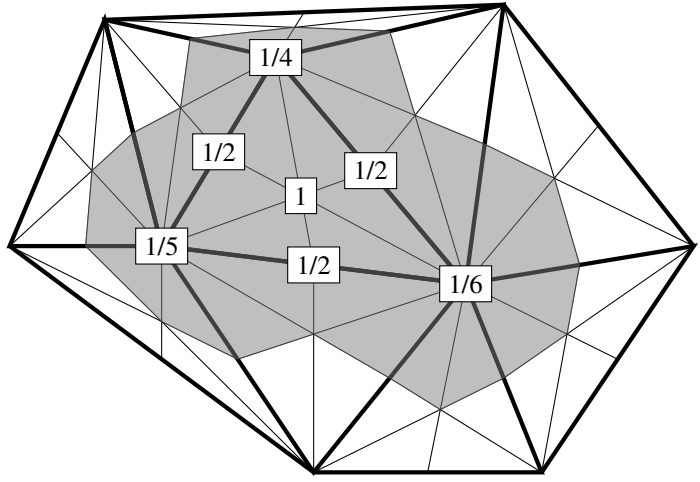
In addition, we require basis functions defined on the dual mesh: on the one hand, we need the so-called \ac{BC} functions  $\wveg f_n \in X_{\wveg f}$, which are the counterpart to the \ac{RWG} functions. Since they do not play a crucial role in this article, we refer the reader to their definition in \cite{buffa_dual_2007}.

On the other hand, we need the dual counterparts to the primal piecewise linear functions $X_{\lambda}$, which have also been defined in \cite{buffa_dual_2007}. %
These dual piecewise linear functions $\wt \lambda_n \in X_{\wt \lambda}$ are attached to the dual vertices, which are located at the barycenter of the cells of the primal mesh.
Similar to $\lambda_n$, the dual piecewise linear functions $\wt \lambda_n$ have to be zero at the boundary of their support. 
In addition, it is required that the dual piecewise linear functions form a partition of unity. Thus the value that each $\wt \lambda_n$ assumes on the primal vertices of the primal cell on which it is defined must be one over the total number of $\wt \lambda_n$ neighboring the respective primal vertex. Likewise, the value that $\wt \lambda_n$ assumes on the midpoints of the edges of the primal cell on which it is defined must be one over the total number of $\wt \lambda_n$ neighboring the respective midpoint (i.e., for closed meshes without junctions the value is always $1/2$).  To illustrate this definition, consider the example in \Cref{fig:DPLF}. The gray area depicts the support of a dual piecewise linear function~$\wt \lambda_n$. The boxes in \Cref{fig:DPLF} show the value of $\wt \lambda_n$  at the respective vertices or midpoints of the primal cell on which $\wt \lambda_n$ is defined on. Between the vertices, midpoints, and the boundary of the support, $\wt \lambda_n$ is linear.

Let $\Omega$ be a closed, perfectly electrically conducting, simply or multiply connected Lipschitz polyhedral domain and the surface $\Gamma = \partial \Omega$ be a triangulation embedded in a space described by the permittivity $\varepsilon$ and the permeability $\mu$.
An electromagnetic wave $\(\veg E^\mr{i}, \veg H^\mr{i}\)$ impinges on $\Omega$ exciting the electric surface current density $\veg j$, which radiates the scattered wave $\(\veg E^\mr{s}, \veg H^\mr{s}\)$ subject to the boundary condition 
\begin{equation}
    \n \times \veg E = \veg 0 \quad\text{for all~}\veg r\in\Gamma
\end{equation}
with the total electric field $\veg E = \veg E^\mr{i} + \veg E^\mr{s}$, Maxwell's equations\\
\noindent\begin{mtabular}{*{2}{m{0.35\linewidth}m{0.15\linewidth}}}
    \begin{equation*}
        \nabla \times \veg E = +\ii \omega \mu   \veg H\, ,
    \end{equation*}
    &
    \eqnCnt \label{eq:Maxwell1}
    &
    \begin{equation*}
        \nabla \times \veg H = -\ii \omega \varepsilon \veg E\, ,
    \end{equation*}
    &
    \eqnCnt  \label{eq:Maxwell2}
\end{mtabular}
and the Silver-Müller radiation condition \cite{silver_microwave_1984,muller_grundzuge_1948}
\begin{equation}
    \lim_{r\rightarrow \infty} \(\eta \veg H^\mr{s} \times \veg r -  r \veg E^\mr{s} \) = 0\, ,
\end{equation}
where $r=\norm{\veg r}$ and $\eta = \sqrt{\mu /\varepsilon}$ is the wave impedance.
The \ac{EFIE} operator
\begin{equation}
    \vecop T \veg  j =\( \im k \TA^k + 1/(\im k ) \TPhi^k\)\veg j\, ,
\end{equation}
where
\begin{equation}
    \TA^\kappa \veg j =  \n \times \int_\Gamma  \frac{\e^{\ii \kappa \vert \veg r - \veg r'\vert}}{4 \uppi\vert \veg r -\veg r'\vert} \veg j(\veg r') \dd S (\veg r')
\end{equation}
is the vector potential operator and
\begin{equation}
    \TPhi^\kappa \veg j =  -\n \times \nabla\int_\Gamma  \frac{\e^{\ii \kappa \vert \veg r - \veg r'\vert}}{4 \uppi \vert \veg r -\veg r'\vert} \nabla_\Gamma'\cdot\veg j(\veg r') \dd S (\veg r')
\end{equation}
is the scalar potential operator and $k$ the wavenumber, relates $\veg E^\mr{i}$ and $\veg j$ by the \ac{EFIE}
\begin{equation}
    \vecop T \veg j = - \n \times \veg E^\mr{i}\, .
\end{equation}
An $\e^{-\ii \omega t}$ time dependency is assumed and suppressed throughout the article.
The radiated wave $\(\veg E^\mr{s}, \veg H^\mr{s}\)$ can be computed by first solving for $\veg j$ and then evaluating the \ac{EFIE} forward operator.
We note that $\veg j$ is normalized with the wave impedance~$\eta$.

For obtaining a numerical solution, we apply the Petrov-Galerkin method, where we use \ac{RWG} functions $\veg f_n \in X_{\veg f}$ as expansion and rotated $\n \times \veg f_n $ as testing functions resulting in the system
\begin{equation}
    \mat T \vec j \coloneq \(\im k \matTA^k+ 1/(\im k) \matTPhi^k \) \vec j= - \vec e\, ,
\end{equation}
where 
\begin{align}
    \matel{\matTA^\kappa}_{mn} &= \(\n \times \veg f_m, \TA^\kappa \veg f_n\)_{\veg L^2}\, , \\ 
    \matel{\matTPhi^\kappa}_{mn} &= \(\n \times \veg f_m, \TPhi^\kappa \veg f_n\)_{\veg L^2}\, , \\
    \vecel{ \vec e}_n &= \(\n \times \veg f_n, \n \times \veg E^\mr{i} \)_{\veg L^2}
\end{align} 
and $\veg j \approx \sum_{n=1}^{N} \vecel{\vec j}_n \veg f_n$, where $\(\veg a, \veg b \)_{\veg L^2} = \int_\Gamma \veg a \cdot \veg b \dd S(\veg r)$ is the $\veg L^2$-inner product and $\veg L^2 = \(L^2\)^3$.
For simplifying the analysis, we deviated in \eqref{eq:RWGdef} from the original definition of the \ac{RWG} functions by not normalizing them with edge length, that is, the relationship to the original \ac{RWG} $\veg f_n^\text{Original}$ as defined in \cite{rao_electromagnetic_1982} is given by $\veg f_n \coloneqq  \veg f_n^\text{Original}/l_n$, where $l_n$ is the length of the $n$th edge.

The matrix $\mat T$ is ill-conditioned both in $k$ and $h$ \cite{christiansen_preconditioner_2002}. 
An optimal preconditioner is given by $\vecop T$ itself: the Calderón identity
\begin{equation}
    \vecop T^2  = -\vecop I/4 + \vecop K^2\, , \label{eq:VectorCalderonIdentity}
\end{equation}
where $\vecop K$ is a compact operator, leading to the (stable) discretization
\begin{equation}
    \mat G_{\n \times \veg f, \wt{\veg f}}^{-\T} \wmat T \mat G_{\n \times \veg f, \wt{\veg f}}^{-1} \mat T \label{eq:DiscreteVectorCalderón}
\end{equation}
is well-conditioned with  $\matel{\wmat T}_{mn} = \(\n \times \wt{\veg f}_m, \vecop T \wt{\veg f}_n\)_{\veg L^2}$, and $\matel{\mat G_{\n \times \veg f, \wt{\veg f}}}_{ij} = \(\n \times \veg f_i,  \wt{\veg f_j}\)_{L^2}$ and $\mat G_{\n \times \veg f, \wt{\veg f}}^{-\T} \coloneqq \(\mat G_{\n \times \veg f, \wt{\veg f}}^{-1}\)^\T$.

The matrix in \eqref{eq:DiscreteVectorCalderón} is, however, not numerically stable down to the static limit since it comprises a null space associated with the harmonic Helmholtz subspace \cite{andriulli_wellconditioned_2013a}. In addition, the excitation $\vec e$ and the unknown current vector $\vec j$ suffer from numerical cancellation.
A first approach to overcome the numerical cancellation could be to use an explicit quasi-Helmholtz decomposition.
While this could succeed in preventing the numerical cancellation and in preserving the quasi-harmonic Helmholtz subspace, it also comes with several drawbacks as will be discussed in the following.

In more detail, let $\veg \Lambda_n \in X_{\veg \Lambda}$ be loop functions,  $\veg H_n \in X_{\veg H}$ be global loops  and $\veg \Sigma_n \in X_{\veg \Sigma}$ be star functions (for a definition of these functions, see for example \cite{wu_study_1995, vecchi_loopstar_1999, wilton_improving_1981, cools_nullspaces_2009}). 
As $X_{\veg f} = X_{\veg \Lambda}\oplus X_{\veg H} \oplus  X_{\veg \Sigma}$, there are transformation matrices $\matL\in \R^{N\times N_\mr{V}}$, $\mat H \in \R^{N\times N_\mr{H}}$, and $\matS\in \R^{N\times N_\mr{C}}$ that link the expansion coefficients of the current in the loop-star basis to the expansion coefficients in the \ac{RWG} basis, that is, we have
\begin{equation}
    \vec j = \matL \vec j_{\veg \Lambda} + \mat H  \vec j_{\veg H} +  \matS \vec j_{\veg \Sigma} \label{eq:qHofj}
\end{equation}
and
\begin{equation}
    \sum_{n=1}^N \vecel{\vec j}_n \veg f_n = \sum_{n=1}^{N_\mr{V}} \vecel{\vec j_{\veg \Lambda}}_n \veg \Lambda_n + \sum_{n=1}^{N_{\mr H}} \vecel{\vec j_{\veg H}}_n \veg H_n  + \sum_{n=1}^{N_\mr{C}}  \vecel{\vec j_{\veg \Sigma}}_n \veg \Sigma_n\, ,
\end{equation} 
where $\vec j_{\veg \Lambda}$, $\vec j_{\veg H}$, and $\vec j_{\veg \Sigma}$ are the unknown vectors in the loop-star basis, and $N_\mr{V}$ is the number of vertices (inner vertices, when $\Gamma$ is an open surface), $N_\mr{C}$ the number of cells, and $N_\mr{H} =2 g$, where $g$ is the genus of $\Gamma$.
The global loops are not uniquely defined. For simplifying the analysis we assume that $\matL^\T \mat H = \vecO$ (such global loops can always be constructed, though at increased computational costs compared with the case where we do not enforce the orthogonality of the transformation matrices $\matL$ and $\mat H$).
We note that loop and star functions are not linearly independent;
the all-one vectors $\matel{ \oL}_n = 1$, $n=1,\dots, N_\mr{V}$, and $\matel{ \oS}_n = 1$, $n=1,\dots, N_\mr{C}$, are in the null spaces of $\matL$ and $\matS$, that is, $\matL \oL = \vecO $ and $\matS \oS = \vecO$.
The linear dependency generates a null space in the transformed system matrix.
To avoid this null space, the classic approach is to eliminate a loop and a star function. For the new formulation that we are presenting here no elimination is necessary.
We define the transformation matrix as $\mat Q \coloneqq \matel{\matL /\sqrt{\ii k} & \mat H /\sqrt{\ii k}  & \matS \sqrt{\ii k} }$. If we were to eliminate loop and star functions so that $\mat Q \in \C^{N \times N}$ has full rank (i.e., this follows the classical loop-star preconditioner approach),
then  $\mat Q^\T \mat T \mat Q$ is well-conditioned in frequency and 
\begin{equation}
    \(\mat Q^\T \mat G_{\n \times \veg f, \wt{\veg f}} \mat Q\)^{-\T} \mat Q^\T \wmat T  \mat Q  \(\mat Q^\T \mat G_{\n \times \veg f, \wt{\veg f}} \mat Q\)^{-1}  \mat Q^\T \mat T \mat Q \label{eq:LSCalderón}
\end{equation}
is well-conditioned in frequency down to the static limit.

There are two drawbacks: 
First, the global loops $\veg H_n$  have to be constructed, which costs, in general, $\mc O(N^2)$ for  $g \asymp N$ and $\mat H$ is dense (a sparse matrix $\mat H$ can be obtained, but then the cost for finding the global loops is $\mc O(N^3)$) (see \cite{adrian_hierarchical_2014a} and references therein).
Second, the Gram matrix $\mat Q^\T \mat G_{\n \times \veg f, \wt{\veg f}} \mat Q$ is ill-conditioned with a condition number that grows as $\mc O(1/h^2)$.
The reason for this is that loop and star functions are not $\veg L^2$-stable, their Gram matrices are equivalent to discretized Laplace-Beltrami operators, for which the condition number grows with $1/h^2$, that is, we have for the loop-loop and the star-star Gram matrix \cite{andriulli_loopstar_2012}
\begin{equation}
    \matL^\T \mat G_{\veg f\veg f} \matL = \mat \Delta \label{eq:PrimalLaplace}
\end{equation} 
and 
\begin{equation}
    \matS^\T \mat G_{\wveg f\wveg f} \matS = \wmat \Delta\, ,\label{eq:DualLaplace}
\end{equation} 
where
\begin{equation}
    \matel{\mat \Delta}_{mn} = \(\nabla_\Gamma \lambda_m, \nabla_\Gamma \lambda_n\)_{L^2}
\end{equation} 
is the Laplace-Beltrami operator discretized with piecewise linear functions $\lambda_n$, and 
\begin{equation}
    \matel{\wmat \Delta}_{mn} = \(\nabla_\Gamma \wt \lambda_m, \nabla_\Gamma \wt\lambda_n\)_{L^2}
\end{equation} 
is the Laplace-Beltrami operator discretized with dual piecewise linear functions $\wt \lambda_n$.
The matrix $\mat Q^\T \mat G_{\n \times \veg f, \wt{\veg f}} \mat Q$ is even more difficult to invert since unlike $\mat \Delta$ it is no longer a symmetric matrix anymore due to $ \mat G_{\n \times \veg f, \wt{\veg f}} $.

Recently a scheme has been presented that leverages on quasi-Helmholtz projectors \cite{andriulli_loopstar_2012,andriulli_wellconditioned_2013a}.
These projectors were defined as 
\begin{equation}
    \Ps \coloneqq \matS \( \matS^\T \matS\)^{+} \matS^\T
\end{equation}
and 
\begin{equation}
    \Plh \coloneqq \matI - \Ps\, ,
\end{equation}
where ``+'' denotes the Moore-Penrose pseudoinverse and $\matI$ is the identity matrix with dimensions fitting to the other matrices in the respective equation.
The projectors are based on the fact that $\matL^\T \matS = \matO$, $\matL^\T \mat H = \matO$, and $\matS^\T \mat H = \matO$.
Hence, $\Ps \vec j$ yields the non-solenoidal part of $\veg j$ expressed in \ac{RWG} expansion coefficients as can  be seen by considering \eqref{eq:qHofj}, that is, $ \Ps \vec j = \matS \vec j_{\veg \Sigma}$ since $\Ps \matS = \matS$ and $\Ps \matL = \matO$ and $\Ps \mat H = \matO$. The subscript of $\Plh$ denotes the fact that it projects to the union of the solenoidal and quasi-harmonic Helmholtz subspace.
For $\wt{\veg f}_n$, dual projectors are defined as 
\begin{equation}
    \Pl \coloneqq \matL \( \matL^\T \matL\)^{+} \matL^\T \label{eq:Pl}
\end{equation}   
and 
\begin{equation}
    \Psh \coloneqq \matI - \Pl \, .
\end{equation}
This allows defining the primal $\mat M \coloneqq \Plh / \sqrt{k} + \im  \Ps \sqrt{k}$ and the dual decomposition operator $\wmat M \coloneqq \Psh / \sqrt{k} + \im  \Pl \sqrt{k}$.
Then the matrix 
\begin{equation}
    \mat G_{\n \times \veg f, \wt{\veg f}}^{-\T} \(\wmat M \wmat T \wmat M \) \mat G_{\n \times \veg f, \wt{\veg f}}^{-1}  \(\mat M \mat T \mat M\) \label{eq:2013}
\end{equation}
is well-conditioned \cite{andriulli_wellconditioned_2013a}.
In contrast to \eqref{eq:LSCalderón}, the costly global loop finding and construction of the (dense) matrix $\mat H$ is avoided. 
Instead of  dealing with the non-symmetric matrix $\mat Q^\T \mat G_{\n \times \veg f, \wt{\veg f}} \mat Q$, in \eqref{eq:2013} only symmetric, positive semi-definite graph Laplacians $\matL^\T \matL$ (vertex-based)  and $\matS^\T \matS$ (cell-based) appear.
As has been pointed out~\cite{andriulli_loopstar_2012,andriulli_wellconditioned_2013a}, a plethora of (black box) algorithms exists for inverting these matrices efficiently.

\section{New Formulation}\label{sec:TheoreticalApparatus}
The formulation for which we are going to show that the system matrix is well-conditioned in the static limit reads
\begin{equation}
    \bPo^\dagger \mat T^\dagger \bPm \mat T \bPo \vec i  = -\bPo^\dagger \mat T^\dagger \bPm \vec e\, , \label{eq:NewFormulation}
\end{equation}
where the outer matrix $\Po$ is
\begin{equation}
    \Po \coloneq \Plh / \sqrt{k} + \im \Pgs \sqrt{k} 
\end{equation} 
with
\begin{equation}
    \Pgs \coloneq \matS \(\matS^\T \matS\)^+ \mat G_{\wt \lambda p}^{-1} \matS^\T\, , \label{eq:Pgs}\\
\end{equation}
and the middle matrix $\Pm$ is
\begin{equation}
    \Pm \coloneq \PmL/ k + \bPmS k
\end{equation}
using
\begin{align}
    \PmL & \coloneq  \matL \mat G_{\lambda \lambda}^{-1} \matL^\T + \Plh\, ,  \label{eq:PmL}\\
    \bPmS  &\coloneq \matS \(\matS^\T \matS\)^+ \mat G_{pp}^{-1} \(\matS^\T \matS\)^+  \matS^\T  \, , 
\end{align}
and with  the conjugate transpose $\Po^\dagger  = \overline{\Po}^\T$. The unknown vector $\vec j$ is recovered by $\vec j  = \Po \vec i $.

The use of the imaginary unit $+\im$ in the definition of $\Po$ is motivated for the same reason as it was for $\mat M$: to prevent the numerical cancellation due to different scaling of the solenoidal and non-solenoidal components in $\vec e$ and $\vec j$ (for a detailed analysis, see \cite{andriulli_wellconditioned_2013a}).

In the new formulation, dual basis functions only appear in the mixed Gram matrix $\mat G_{\wt \lambda p}$. For this matrix, we have the analytic formula
\begin{equation}
    \matel{\mat G_{\wt \lambda p}}_{mn} = \begin{cases}
        \dfrac{2}{18} \(\dfrac{9}{2} + \displaystyle\sum_{i=1}^3 \dfrac{1}{\mathit{NoC}(\mathit{VoC}(m,i))}  \) \quad & \text{if $m=n$}\, ,\\
        \dfrac{2}{18} \(\dfrac{1}{2} + \dfrac{1}{\mathit{NoC}(e^+)} + \dfrac{1}{\mathit{NoC}(e^-)} \) \quad &\text{if cells $m$ and $n$ share edge $e$}\, ,\\ 
        \dfrac{2}{18} \(\dfrac{1}{\mathit{NoC}(v)} \) \quad &\text{if cells $m$ and $n$ are only connected by vertex $v$}\, ,\\ 
        0 \quad &\text{otherwise,}
    \end{cases}
\end{equation}
where the function $\mathit{NoC}(v)$ returns the number of cells attached to $v$th vertex of the mesh, the function $\mathit{VoC}(c,i)$ returns the global index of the $i$th vertex of the $c$th cell, and $e^+$ and $e^-$ are the indices of the vertices of the $e$th edge.

For proving that the system matrix in \eqref{eq:NewFormulation} is well-conditioned, we first decompose the \ac{EFIE} into two scalar operators.
In the static limit $k \rightarrow 0$ and for simply connected geometries (i.e., the quasi-harmonic Helmholtz subspace is not present), we have the well-established equality (see, for example, a more detailed derivation in \cite{adrian_hierarchical_2017})
\begin{equation}
    \lim_{k \rightarrow 0} \mat Q^\T  \mat T \mat Q 
    = \begin{bmatrix}
        \matL^\T \matTA^0 \matL & \vecO \\
        \vecO & \matS^\T \matTPhi^0 \matS 
    \end{bmatrix}
    = \begin{bmatrix}
        \mat W & \vecO \\
        \vecO & \matS^\T \matS \mat V \matS^\T\matS 
    \end{bmatrix}\, , \label{eq:StaticLS}
\end{equation}
with
\begin{equation}
    \matel{\mat W}_{mn} = \(\lambda_m, \op W \lambda_n \)_{L^2} \label{eq:DefW}
\end{equation}
and
\begin{equation}
    \matel{\mat V}_{mn} = \(p_m, \op V p_n \)_{L^2}\, ,\label{eq:DefV}
\end{equation}
where
\begin{equation}
    \op W \lambda = \n_{\veg r} \cdot \nabla_\Gamma \times \int_\Gamma \frac{1}{4 \uppi \abs{\veg r - \veg r'}} \nabla_\Gamma' \times \n_{\veg r'} \lambda(\veg r') \dd S(\veg r') \label{eq:W}
\end{equation}
is the hypersingular operator, and
\begin{equation}
    \op V p = \int_\Gamma \frac{1}{4 \uppi \abs{\veg r - \veg r'}} p(\veg r') \dd S(\veg r')
\end{equation}
is the single layer operator; both operators are well-known from electrostatics.
A detailed derivation of the equality 
\begin{equation}
    \matTPhi^0 = \matS \mat V \matS^\T \label{eq:TPhiandV}
\end{equation} 
used in \eqref{eq:StaticLS} can be found in Section~3.1 in \cite{adrian_hierarchical_2017}.
From \eqref{eq:StaticLS}, \eqref{eq:DefW}, and \eqref{eq:DefV}, we can see that $\vecop T$ discretized with the loop-star basis (where the loop and the star functions need to be appropriately rescaled in $k$) decouples into the scalar operators $\op W$ and $\op V$ in the static limit, where the latter is accompanied by the graph Laplacian $\matS^\T \matS$.
Given that this graph Laplacian can be removed by using its (pseudo-)inverse, it remains the task to find preconditioners for $\mat W$ and for $\mat V$.

In order to show the well-conditioning of \eqref{eq:NewFormulation} in the static limit, we need to establish spectral equivalences between $\mat W$ and $\mat \Delta$ as well as $\mat V$ and $\wmat \Delta$.
These equivalences will be established by using Rayleigh quotients: we call two symmetric, positive definite matrices $\mat A, \mat B \in \R^{n\times n}$ spectrally equivalent if they satisfy
\begin{equation}
    \vec x^\T \mat A \vec x \asymp \vec x^\T \mat B \vec x \quad\text{for all~} \vec x \in \R^n\, .
\end{equation}
An immediate implication of this inequality is that 
\begin{equation}
    \cond(\mat B^{-1} \mat A)  < C\, ,
\end{equation}
where $C \in \R^+$ is a constant independent from $h$.

What makes it difficult is that for example $\mat W$ possesses a null space and if we need to form the inverse of a product of matrices where some matrices have a null space and some matrices do not, then the inverse of such a product cannot be simplified easily. To avoid the null space issue, we follow a standard approach by defining operators that are identical to $\op W$ and $\Laplace_\Gamma$ for mean-value free functions, but have no null space \cite{steinbach_numerical_2010}:
we introduce the operator $\hat{\op W} : H^{1/2} \rightarrow H^{-1/2}$ defined by the bilinear form
\begin{equation}
    \(v, \hat{\op W}w\)_{L^2(\Gamma)} \coloneqq \(v, \op W w\)_{L^2(\Gamma)}  + \(1, w\)_{L^2(\Gamma)}  \(1, v\)_{L^2(\Gamma)}  \label{eq:DeflectedW}
\end{equation}
for all $w,v \in H^{1/2}(\Gamma)$. 
We note that the unique solution $w$ of 
\begin{equation}
    \hat{\op W}w = g \label{eq:hatWg}
\end{equation}
is also a solution of 
\begin{equation}
    \op W w = g
\end{equation}   
when $g$ satisfies the solvability condition $\int_\Gamma g \dd S(\veg r') = 0$.
This can be seen when $v=1$ in \eqref{eq:DeflectedW} and we find that $(1, \hat{\op W} w)_{L^2(\Gamma)} =  (1, g)_{L^2(\Gamma)}$ reduces to $\(1,w\)_{L^2(\Gamma)}  \(1,1\)_{L^2(\Gamma)}=0$, which means that the solution of $\eqref{eq:hatWg}$ is mean-value if $g$ is so.
Likewise, if we let $\Laplace_\Gamma : H^1/\R \rightarrow H^{-1}$ be the Laplace-Beltrami operator~\cite{nedelec_acoustic_2001}, then we consider the $H^1$-elliptic modified Laplace-Beltrami operator $\hat{\Laplace}_\Gamma$ to be defined by the bilinear form
\begin{equation}
    \(v, -\hat{\Laplace}_\Gamma w\)_{L^2(\Gamma)} \coloneqq \(\nabla_\Gamma w, \nabla_\Gamma v\)_{L^2(\Gamma)} +  \(1, w\)_{L^2(\Gamma)}  \(1, v\)_{L^2(\Gamma)}  \label{eq:DeflectedLaplacian}\, .
\end{equation}
When $\lambda_i$ are used for the discretization of $\hat{\op W}$ and $\hat \Laplace_\Gamma$, the resulting matrices are
\begin{equation}
    \hmat W = \mat W + \mat G_{\lambda\lambda}^\T \oL \oL^\T  \mat G_{\lambda\lambda} \label{eq:DiscreteModW}
\end{equation}
and
\begin{equation}
    \hmat \Delta = \mat \Delta + \mat G_{\lambda\lambda}^\T \oL \oL^\T  \mat G_{\lambda\lambda} \, .
\end{equation}
If  $\wt \lambda_i$ are used  for the discretization, the resulting matrices read
\begin{equation}
    \htmat W = \wmat W + \mat G_{\wt\lambda\wt\lambda}^\T \oS \oS^\T  \mat G_{\wt\lambda\wt\lambda} 
\end{equation}
and
\begin{equation}
    \htmat \Delta = \wmat \Delta + \mat G_{\wt\lambda\wt\lambda}^\T \oS \oS^\T  \mat G_{\wt\lambda\wt\lambda} \, .
\end{equation}
In addition, we need to define
\begin{equation}
    \cmat W \coloneqq \mat W + \oL \oL^\T h^4
\end{equation}
and
\begin{equation}
    \cmat \Delta \coloneqq \mat \Delta + \oL \oL^\T  h^4 \, .
\end{equation}

In the following, we assume the spectral equivalences
\begin{equation}
    \vec x^\T \hmat W \mat G_{\lambda \lambda}^{-1} \hmat W\vec x \asymp \vec x^\T \hmat \Delta \vec x  \quad\text{for all~} \vec x \in \R^{N_\mr{V}} \label{eq:WWDelta}
\end{equation}
and 
\begin{equation}
    \vec x^\T \htmat W \mat G_{\wt \lambda \wt\lambda}^{-1} \htmat W \vec x \asymp \vec x^\T \htmat \Delta \vec x  \quad\text{for all~} \vec x \in \R^{N_\mr{C}}\, . \label{eq:WWdualDelta}
\end{equation}

\begin{remark}
    As shown in the Appendix for \eqref{eq:WWDelta}, such a spectral equivalence can be established when we have a nested sequence of function spaces by using a hierarchical basis. Thus in this section, we assume a structured mesh refinement, which gives rise to a nested sequence of piecewise constant and a piecewise linear function space, respectively.
    We note that the structured refinement implies a quasi-uniform family of meshes, that is
    \begin{equation}
        h_\text{max}/h_\text{min} < c_\text{Q}\, , 
    \end{equation}
    where $h_\text{max}$ is the length of the largest edge, $h_\text{min}$ is the length of the smallest edge and $c_\text{Q}$ is a constant independent from $h$.
    Similarly to preconditioning strategies such as algebraic multigrid, the new formulation remains effective even on unstructured meshes as can be seen from the numerical results in \Cref{sec:NumericalResults}.
    Moreover, we note that $\mat G_{\lambda \lambda}^{-1} $ could be replaced by $ h^2\, \mat I$. By using  $\mat G_{\lambda \lambda}^{-1} $, however, the bounding constants in \eqref{eq:WWDelta} are typically sharper when the bounding constant $c_\text{Q}$ is large and thus the overall condition number is reduced.
\end{remark}

For proving the well-conditionedness of new formulation, we start with making the frequency dependency explicit by considering
\begin{multline}
    \bPo^\dagger \mat T^\dagger \bPm \mat T \bPo = \Plh^\dagger \(\matTA^k\)^\dagger \PmL \matTA^k \Plh + \Pgs^\dagger \(\matTPhi^k\)^\dagger \bPmS \matTPhi^k \Pgs \\
    + k^2\Pgs^\dagger \(\matTA^k\)^\dagger \PmL \matTA^k \Pgs + k^4 \Pgs^\dagger  \(\matTA^k\)^\dagger \bPmS \matTA^k \Pgs \\
    +\im k \Plh^\dagger \(\matTA^k\)^\dagger \PmL \matTA^k \Pgs
    -\im k \Pgs^\dagger \(\matTA^k\)^\dagger \PmL \matTA^k \Plh \\
    + \im k^3 \Plh^\dagger \(\matTA^k\)^\dagger \bPmS \matTA^k \Pgs 
    -\im k^3  \Pgs^\dagger \(\matTA^k\)^\dagger \bPmS \matTA^k \Plh \\
    - \im k \Plh^\dagger \(\matTA^k\)^\dagger \bPmS \matTPhi^k \Pgs 
    + \im k\Pgs^\dagger \(\matTPhi^k\)^\dagger \bPmS \matTA^k \Plh \, .\label{eq:LastProofBigUmformung}
\end{multline}
Thus we find in the static limit
\begin{equation}
    \lim_{k \rightarrow 0}     \bPo^\dagger \mat T^\dagger \bPm \mat T \bPo  =  \Plh^\dagger \(\matTA^0\)^\dagger \PmL \matTA^0 \Plh + \Pgs^\dagger \(\matTPhi^0\)^\dagger \bPmS \matTPhi^0 \Pgs\, .
\end{equation}
Due to the orthogonality of the null spaces of $\Plh$ and $\Pgs$, the matrix $\bPo^\dagger \mat T^\dagger \bPm \mat T \bPo $ is well-conditioned if $ \Plh^\dagger \(\matTA^0\)^\dagger \PmL \matTA^0 \Plh$ and $\Pgs^\dagger \(\matTPhi^0\)^\dagger \bPmS \matTPhi^0 \Pgs$ are well-conditioned, respectively. The well-conditionedness of the latter two matrices is proved in \Cref{sec:VectorPotential} and \Cref{sec:ScalarPotential}.

Before proving the well-conditionedness rigorously, we like to give a heuristic argument by assigning orders to the matrices corresponding to the order of the underlying pseudo-differential operator. Thus, we have that $\matTA$ is of order $-1$ and $\matTPhi$,  $\matL$, and $\matS$ are each of order $+1$. The Gram matrices, $\Plh$, and $\Pgs$ are of order $0$. By simply counting the orders, we note that the total order of $ \Plh^\dagger \(\matTA^0\)^\dagger \PmL \matTA^0 \Plh$ and of $\Pgs^\dagger \(\matTPhi^0\)^\dagger \bPmS \matTPhi^0 \Pgs$, respectively, is~$0$. Consequently, the order of the sum of these two matrices is also~$0$. Since pseudo-differential operators of order~$0$ result in a well-conditioned matrix if an $L^2$-stable basis has been used for their discretization, this suggests that the new formulation is well-conditioned.

\subsection{Vector Potential}\label{sec:VectorPotential}
\subsubsection{Simply Connected Case}
Here we prove that 
\begin{equation}
    \Pl \(\matTA^k\)^\dagger \matL \mat G_{\lambda \lambda}^{-1} \matL^\T \matTA^k \Pl 
\end{equation}
is well-conditioned  up to its null space.
To this end, we first need to establish several spectral equivalences.

\begin{lem}  \label{lem:HatandCheckEquivs}
    We have the spectral equivalences
    \begin{equation}
        \vec x^\T \hmat \Delta  \vec x \asymp  \vec x^\T  \cmat \Delta \vec x\quad\text{for all~} \vec  x \in \R^{N_\mr{V}} \label{eq:hDeltaEQcDelta}
    \end{equation}   
    and 
    \begin{equation}
        \vec x^\T \hmat W \vec x \asymp \vec x^\T  \cmat W \vec x \quad\text{for all~} \vec  x \in \R^{N_\mr{V}} \, . \label{eq:hWEQcW}
    \end{equation}   
\end{lem}
\begin{proof}
    Here we prove \eqref{eq:hDeltaEQcDelta}; the proof for \eqref{eq:hWEQcW} follows analogously.
    We note that the null space of $\mat \Delta $ is spanned by $\oL$.
    Furthermore, we note 
    \begin{equation}
        \oL^\T \mat G_{\lambda\lambda}^\T \oL \oL^\T \mat G_{\lambda\lambda}  \oL = A_\Gamma^2 \asymp 1\, , \label{eq:ScalingCorrDefl}
    \end{equation}
    where $A_\Gamma$ is the area of $\Gamma$,
    and $\norm{\oL}_2 = \sqrt{N_\mr{V}} \asymp 1/h$ and thus 
    \begin{equation}
        \oL^\T \oL \oL^\T \oL h^4 \asymp 1\, . \label{eq:ScalingWrongDefl}
    \end{equation}
    
    First, consider that we have
    \begin{equation}
        \vec x^\T \oL \oL^\T h^4  \vec x \lesssim \vec x^\T \vec x h^2 \quad\text{for all~} \vec  x \in \R^{N_\mr{V}} \label{eq:UpperBoundOneDefl}
    \end{equation}
    and
    \begin{equation}
        \vec x^\T \mat G_{\lambda\lambda}^\T \oL \oL^\T \mat G_{\lambda\lambda}  \vec x \lesssim  \vec x^\T \vec x h^2 \quad\text{for all~} \vec  x \in \R^{N_\mr{V}}\, , \label{eq:UpperBoundIntegralDefl}
    \end{equation}
    where the last inequality follows from the well-known equivalence
    \begin{equation}
        \vec x^\T \mat G_{\lambda\lambda} \vec x \asymp \vec x^\T  \vec x h^2 \quad\text{for all~} \vec  x \in \R^{N_\mr{V}}\, , \label{eq:Nodall2Stability}
    \end{equation}
    the submultiplicativity of the matrix norm and \eqref{eq:ScalingCorrDefl}, that is,
    \begin{equation}
        \norm{\mat G_{\lambda\lambda}^\T \oL \oL^\T \mat G_{\lambda\lambda}}_2 \leq \norm{\mat G_{\lambda\lambda}^\T}_2 \norm{\oL \oL^\T}_2 \norm{\mat G_{\lambda\lambda}}_2 \lesssim h^2 \, .
    \end{equation}
    
    Let $\vec x = \vec x_\parallel + \vec x_\perp$ be an orthogonal splitting with $\vec x_\parallel \in \Span \oL $. If $\vec x_\perp = \ooL $, then 
    \begin{equation}
        \vec x^\T \hmat \Delta \vec x = \vec x^\T \mat G_{\lambda\lambda}^\T \oL \oL^\T \mat G_{\lambda\lambda}  \vec x \asymp \vec x^\T \oL \oL^\T h^4 \vec x =  \vec x^\T \cmat \Delta \vec x
    \end{equation}
    due to  \eqref{eq:ScalingCorrDefl}, \eqref{eq:ScalingWrongDefl} (noting that $ \vec x_\parallel$ is a multiple of $\oL$). Since we have \cite{andriulli_loopstar_2012}
    \begin{equation}
        \vec x^\T  \vec x  h^2 \lesssim\vec x^\T \mat \Delta \vec x \lesssim \vec x^\T  \vec x\quad\text{for all~} \vec x \in \(\Span \oL \)^\perp\, ,
    \end{equation}  
    we note that for $\vec x_\perp \neq \ooL$ the leading contribution $\vec x^\T \mat \Delta \vec x$ scales at least $\mc O(h^2)$ and at most $\mc O(1)$.  The contribution from  $\vec x^\T \mat G_{\lambda\lambda}^\T \oL \oL^\T \mat G_{\lambda\lambda}  \vec x$ and $\vec x^\T \oL \oL^\T \vec x$ adds a positive quantity that scales at most $\mc O(h^2)$ due to \eqref{eq:UpperBoundIntegralDefl}, \eqref{eq:Nodall2Stability}, and the fact that $\mat G_{\lambda\lambda}^\T \oL \oL^\T \mat G_{\lambda\lambda} $ and $\oL \oL^\T $ are positive semi-definite rank\hbox{-}1 matrices. Hence for $h\rightarrow 0$, we can conclude that the eigenvalues of $\hmat \Delta$ and $\cmat \Delta$ which scale with $\mc O(h^\alpha)$, $0 \leq \alpha < 2$ (and their associated eigenvectors) are spectrally identical and the eigenvalues scaling by $\mc O(h^2)$ are shifted at most by a constant factor. Thus \eqref{eq:hDeltaEQcDelta} follows.
    
    For \eqref{eq:hWEQcW}, the same argumentation can be used noting that
    \begin{equation}
        \vec x^\T  \vec x  h^2 \lesssim\vec x^\T \mat W \vec x \lesssim \vec x^\T  \vec x h \quad\text{for all~} \vec x \in \(\Span \oL\)^\perp \label{eq:WSpectral}
    \end{equation}  
    holds.
\end{proof}

\begin{remark}
    This lemma will be frequently used in order to replace $\mat G_{\lambda\lambda}^\T \oL \oL^\T \mat G_{\lambda\lambda} $ by $\oL \oL^\T $. In essence, we are allowed to do so if the accompanying matrix has a null space spanned by $\oL$ and where the smallest non-zero singular value scales at most quadratically in $h$.
\end{remark}

\begin{lem}\label{lem:ModifiedWGWEquiv}
    We have the spectral equivalences
    \begin{equation}
        \vec x^\T \hmat W  \mat  G_{\lambda \lambda}^{-1} \hmat W \vec x \asymp \vec x^\T \(\mat W\mat  G_{\lambda \lambda}^{-1} \mat W  +  \mat G_{\lambda\lambda}^\T \oL \oL^\T \mat G_{\lambda\lambda} \) \vec x 
        \asymp \vec x^\T \( \mat W\mat  G_{\lambda \lambda}^{-1} \mat W  +  \oL \oL^\T h^4  \) \vec x\quad\text{for all~} \vec x \in \R^{N_\mr{V}}  \, . \label{eq:hWhWWW}
    \end{equation}   
\end{lem}
\begin{proof}
    We have 
    \begin{equation}
        \hmat W  \mat  G_{\lambda \lambda}^{-1} \hmat W  = \( \mat W \mat G_{\lambda \lambda}^{-1} + \mat G_{\lambda \lambda} \oL \oL^\T \)\( \mat W +  \mat G_{\lambda \lambda} \oL \oL^\T \mat G_{\lambda \lambda}   \) 
        = \mat W \mat G_{\lambda \lambda}^{-1} \mat W +  \mat G_{\lambda \lambda} \oL \oL^\T  \mat G_{\lambda \lambda} \oL \oL^\T \mat G_{\lambda \lambda} 
    \end{equation}
    using $\mat W \oL = \vecO$ and $ \oL^\T \mat W= \vecO^\T$.
    Since $\oL^\T  \mat G_{\lambda \lambda} \oL = \int_{\Gamma} \dd S(\veg r) = A_\Gamma$ is a constant,  we yield
    \begin{equation}
        \vec x^\T \hmat W  \mat  G_{\lambda \lambda}^{-1} \hmat W  \vec x 
        \asymp  \vec x^\T \(\mat W \mat G_{\lambda \lambda}^{-1} \mat W +  \mat G_{\lambda \lambda} \oL  \oL^\T \mat G_{\lambda \lambda} \)\vec x \quad\text{for all~} \vec x \in \R^{N_\mr{V}}\, ,
    \end{equation}
    which proves the first equivalence in \eqref{eq:hWhWWW}.
    For the second equivalence, we have due to \eqref{eq:Nodall2Stability}
    \begin{equation}
        \vec x^\T \mat W \mat G_{\lambda \lambda}^{-1} \mat  W \vec x \asymp \vec x^\T \mat W \mat  W \vec x  / h^2\quad\text{for all~} \vec x \in \R^{N_\mr{V}}
    \end{equation}
    and using \eqref{eq:WSpectral}, we have
    \begin{equation}
        \vec x^\T  \vec x  h^2 \lesssim\vec x^\T \mat W \mat W \vec x /  h^2 \lesssim \vec x^\T  \vec x  \quad\text{for all~} \vec x \in \(\Span \oL \)^\perp \, .
    \end{equation}  
    Then using the argumentation of \Cref{lem:HatandCheckEquivs}, we obtain
    \begin{equation}
        \vec x^\T \(\mat W\mat  G_{\lambda \lambda}^{-1} \mat W  +  \mat G_{\lambda\lambda}^\T \oL \oL^\T \mat G_{\lambda\lambda} \) \vec x 
        \asymp \vec x^\T \( \mat W\mat  G_{\lambda \lambda}^{-1} \mat W  +  \oL \oL^\T h^4  \) \vec x\quad\text{for all~} \vec x \in \R^{N_\mr{V}} \, .
    \end{equation}
\end{proof}
\begin{lem}\label{lem:CLaplacianGraph}
    We have the spectral equivalence
    \begin{equation}
        \vec x^\T \cmat \Delta  \vec x \asymp \vec x^\T \( \matL^\T \matL + \oL \oL^\T h^4 \) \vec x\quad\text{for all~}\vec x \in \R^{N_\mr{V}}\, .
    \end{equation}
\end{lem}
\begin{proof}
    From \eqref{eq:PrimalLaplace} and  \cite{andriulli_loopstar_2012}
    \begin{equation}
        \vec x^\T \mat G_{\veg f \veg f} \vec x  \asymp  \vec x^\T \mat G_{\wveg f \wveg f} \vec x  \asymp \vec x^\T \vec x \quad\text{for all~} \vec x \in \R^{N}\, , \label{eq:WellGff}
    \end{equation}
    we obtain
    \begin{equation}
        x^\T \mat \Delta \vec x \asymp \vec x^\T  \matL^\T \matL \vec x\quad\text{for all~} \vec x \in \R^{N_\mr{V}}\, . \label{eq:Beltrami_Graph_Equiv}
    \end{equation}
    Equation \eqref{eq:Beltrami_Graph_Equiv} remains true when we add $\oL \oL^\T h^4$ to the matrices since all matrices appearing are positive, semi-definite.
\end{proof}

\begin{prop}\label{prop:WellconditionedVectorPot}
    We have the spectral equivalence
    \begin{equation}
        \vec x^\T \Pl \(\matTA^k\)^\dagger \matL \mat G_{\lambda \lambda}^{-1} \matL^\T \matTA^k \Pl \vec x \asymp \vec x^\T \Pl \vec x\quad\text{for all~} \vec x \in \R^{N}\, , \label{eq:WellConditionedVectorPot}
    \end{equation}
    with $\Pl \coloneqq  \matL \( \matL^\T \matL\)^{+}\matL^\T$.
\end{prop}
\begin{proof}
    By combining the previous lemmas, we can establish
    \begin{equation}
        \vec x^\T \( \mat W\mat  G_{\lambda \lambda}^{-1} \mat W  +  \oL \oL^\T h^4 \) \vec x 
        \asymp  \vec x^\T \(\matL^\T \matL + \oL \oL^\T h^4 \) \vec x\quad\text{for all~} \vec x \in \R^{N_\mr{V}}\, ,
    \end{equation}
    that is in more detail, we have
    \begin{multline}
        \vec x^\T \( \mat W\mat  G_{\lambda \lambda}^{-1} \mat W  +  \oL \oL^\T h^4 \) \vec x  \stackrel{\text{\Cref{lem:ModifiedWGWEquiv}}}{\asymp}     \vec x^\T \hmat W  \mat  G_{\lambda \lambda}^{-1} \hmat W \vec x  \\ \stackrel{\text{\Cref{lem:SpectralEquivance}}}{\asymp}       \vec x^\T \hmat \Delta  \vec x 
        \stackrel{\text{\Cref{lem:HatandCheckEquivs}}}{\asymp}   \vec x^\T \cmat \Delta  \vec x 
        \\  \stackrel{\text{\Cref{lem:CLaplacianGraph}}}{\asymp}     \vec x^\T \(\matL^\T \matL + \oL \oL^\T h^4 \) \vec x\quad\text{for all~} \vec x \in \R^{N_\mr{V}}\, .
    \end{multline}
    We apply the substitution $\vec x = \(\matL^\T \matL \)^+ \matL^ \T \vec y$ and obtain
    \begin{equation}
        \vec y^\T   \matL \(\matL^\T \matL \)^+ \mat W\mat  G_{\lambda \lambda}^{-1} \mat W  \(\matL^\T \matL \)^+ \matL^\T \vec y \asymp  \vec y^\T \Pl \vec y\quad\text{for all~} \vec y \in \R^{N}\, , \label{eq:TAComparison}
    \end{equation} 
    where we used $\Pl \equiv   \matL \( \matL^\T \matL\)^+  \matL^\T$.
    We note that $ \matL^\T \matTA^0 \matL = \mat W = \mat W^\dagger$;
    for $k \neq 0$, the dynamic kernel introduces a compact perturbation \cite{hiptmair_natural_2002}: the operator $\TA^k - \TA^0$. Such a perturbation may deteriorate the condition number for a given frequency $k$, but since it is compact, the eigenvalues of the discrete counterpart accumulate at  zero for  $h\rightarrow 0$. Hence we still have an $h$-independent upper bound for the condition number. This allows us  to choose $\matTA^\dagger$ instead of  $\matTA$ resulting in \eqref{eq:WellConditionedVectorPot}. 
\end{proof}
\begin{remark}
    The matrix $\Pl \(\matTA^k\)^\dagger \matL \mat G_{\lambda \lambda}^{-1} \matL^\T \matTA^k \Pl  $ is Hermitian and positive semi-definite, which can be seen by considering
    \begin{equation}
        \Pl \(\matTA^k\)^\dagger \matL \mat G_{\lambda \lambda}^{-1} \matL^\T \matTA^k \Pl  = \( \mat G_{\lambda \lambda}^{-1/2} \matL^\T \matTA^k \Pl  \)^\dagger \mat G_{\lambda \lambda}^{-1/2} \matL^\T \matTA^k \Pl\, .
    \end{equation}
\end{remark}

\subsubsection{Multiply Connected Case}
Now, we are ready to consider the case that $\Gamma$ is multiply connected:
we have to establish that 
\begin{equation}
    \Plh  \(\matTA^k\)^\dagger \( \matL \mat G_{\lambda \lambda}^{-1} \matL^\T + \Plh \) \matTA^k \Plh \label{eq:TAonMP}
\end{equation}
is well-conditioned up to its null space. For deriving this result, some preliminary considerations are necessary, and again, we start with considering the static limit $k\rightarrow 0$.

\begin{prop}\label{prop:WellconditionedExplicitTA}
    We have the spectral equivalence
    \begin{equation}
        \vec x^\T     \(\bmatL^\T \bmatL \)^{-1/4} \bmatL^\T \matTA^0 \bmatL \(\bmatL^\T \bmatL \)^{-1/4} \vec x 
        \asymp \vec x^\T  \vec x\quad\text{for all~} \vec x \in \(\Span \oL \)^\perp \, ,
    \end{equation}
    where we used the substitution $\bmatL \coloneqq \matL/h$.
\end{prop}
\begin{proof}
    \Cref{prop:WellconditionedVectorPot} implies that
    \begin{equation}
        \vec x^\T \Pl \matTA^0 \matL \mat G_{\lambda \lambda}^{-1} \matL^\T \matTA^0 \Pl \vec x \asymp \vec x^\T  \Pl \vec x\quad\text{for all~} \vec x \in \mel X_\upLambda
    \end{equation}
    with
    \begin{equation}
        \mel X_\upLambda \coloneqq \{\vec x \in \R^n  \mid  \vec x = \Pl \vec x  \}
    \end{equation}
    holds. In what follows, the expression $\(\matL^\T \matL\)^{-1/2}$ has to be read as $\(\(\matL^\T \matL\)^{+}\)^{+1/2}$. We apply the substitution $\vec y =  (\matL^\T \matL )^{-1/2}  \matL^\T \vec x$ noting that $(\matL^\T \matL )^{-1/2}  \matL^\T  : \mel X_\upLambda \rightarrow  \(\Span \oL \)^\perp $ is one-to-one and onto and that 
    \begin{equation}
        \Pl = \matL (\matL^\T \matL )^{-1/2}  (\matL^\T \matL )^{-1/2}  \matL^\T
    \end{equation}
    so we obtain
    \begin{equation}
        \vec y^\T  \(\matL^\T \matL \)^{-1/2} \matL^\T \matTA^0 \matL \mat G_{\lambda \lambda}^{-1} \matL^\T \matTA^0  \matL \(\matL^\T \matL \)^{-1/2}  \vec y  \asymp \vec y^\T  \vec y\quad\text{for all~} \vec y \in \(\Span \oL \)^\perp\, .
    \end{equation}
    By using \eqref{eq:Nodall2Stability}, we obtain
    \begin{equation}
        \vec x^\T \(\matL^\T \matL \)^{-1/2} \matL^\T \matTA^0 \matL \(h^{-2} \) \matL^\T \matTA^0 \matL \(\matL^\T \matL \)^{-1/2} \vec x \asymp \vec x^\T \vec x  \quad\text{for all~} \vec x \in  \(\Span \oL \)^\perp
    \end{equation}
    We define $\bmatL = \matL /h$ noting that
    \begin{equation}
        \(\matL^\T \matL \)^{-1/2} \matL^\T \matTA^0\ \matL \(h^{-2} \) \matL^\T \matTA^0 \matL \(\matL^\T \matL \)^{-1/2} 
        =   \(\bmatL^\T \bmatL \)^{-1/2} \bmatL^\T \(\matTA^0\)^\dagger \bmatL \bmatL^\T \matTA^0 \bmatL \(\bmatL^\T \bmatL \)^{-1/2} \, .
    \end{equation}
    Furthermore, we notice that the singular values of  
    \begin{equation}
        \bmatL^\T \matTA^0 \bmatL \(\bmatL^\T \bmatL \)^{-1/2} \label{eq:SqrtMatrixOfWellTA}
    \end{equation}
    are by definition of the \ac{SVD} the square roots of the singular values of 
    \begin{equation}
        \(\bmatL^\T \bmatL \)^{-1/2} \bmatL^\T \matTA^0 \bmatL  \bmatL^\T \matTA^0 \bmatL \(\bmatL^\T \bmatL \)^{-1/2}\, ,
    \end{equation}
    which implies that the matrix in \eqref{eq:SqrtMatrixOfWellTA} is well-conditioned. 
    
    The absolute value of the largest eigenvalue can always be bounded from above by the largest singular value and the smallest eigenvalue can always be bounded from below by the smallest singular value. The second half of this statement is not entirely helpful since the smallest eigenvalue and singular value are both zero. However, since the left null space and the right null space of \eqref{eq:SqrtMatrixOfWellTA}  are identical, we can show that the smallest non-zero absolute eigenvalue $\nu_\mr{min}$ can be bounded from below by the smallest non-zero singular value $s_\mr{min}$ .
    
    To see this, let $\vec v$ be the unit eigenvector associated with $\nu_\mr{min}$ and use the abbreviation $\mat A = \bmatL^\T \matTA^0 \bmatL (\bmatL^\T \bmatL )^{-1/2}$. We have
    \begin{equation}
        \abs{\nu_\mr{min}}^2  = \vec v^\T \mat A^\T \mat A \vec v  \geq \min_{\norm{\vec x}_2 =1 \wedge  \vec x \perp \oL} \vec x^\T \mat A^\T \mat A \vec x = s_\mr{min}^2
    \end{equation}
    following the properties of the \ac{SVD} and noting that $\norm{\vec v}_2 =1 $ and $\vec v \perp \oL$.
    
    Similar matrices have the same eigenvalues and thus $\mat A$ and $(\matL^\T \matL )^{-1/4} \mat A (\matL^\T \matL )^{1/4}$ have the same eigenvalues. Since
    \begin{equation}
        \(\bmatL^\T \bmatL \)^{-1/4} \mat A \(\bmatL^\T \bmatL \)^{1/4} = \(\bmatL^\T \bmatL \)^{-1/4} \bmatL^\T \matTA^0 \bmatL \(\bmatL^\T \bmatL \)^{-1/4}
    \end{equation}
    is a symmetric, positive semidefinite matrix, the eigenvalues and singular values coincide and thus
    \begin{equation}
        \(\bmatL^\T \bmatL \)^{-1/4} \bmatL^\T \matTA^0 \bmatL \(\bmatL^\T \bmatL \)^{-1/4}
    \end{equation}
    is well-conditioned up to its null space.
\end{proof}
Given \Cref{prop:WellconditionedExplicitTA}, we can conclude that if the matrix $\bmatL (\bmatL^\T \bmatL )^{-1/4}$ is used as solenoidal basis, we obtain a well-conditioned matrix with bounded norm. Thus if we were to pursue a classical \emph{explicit} quasi-Helmholtz decomposition scheme, we could use the basis $\begin{bmatrix} \bmatL (\bmatL^\T \bmatL )^{-1/4} & \mat H \end{bmatrix}$ as preconditioner for $\matTA$ on multiply connected geometries. 

We note that $\mat H^\T \matTA^0 \mat H$ is well-conditioned since the global loops are associated with the geometry of $\Gamma$; therefore, $\veg H_n$ and subsequently $\(\veg H_m, \TA^\kappa \veg H_n \)_{L^2}$ remain the same when $h \rightarrow 0$, since we assumed a structured mesh refinement, so that we have nested sequences of function spaces.
Thus we have
\begin{equation}
    \vec x^\T \mat H^\T \matTA^0 \mat H \vec x \asymp \vec x^\T  \mat H^\T \mat H \vec x \asymp \vec x^\T \vec x\quad\text{for all~} \vec x \in \R^{N_\mr{H}} \, . \label{eq:MultiplyConnectedHTAH}
\end{equation}
Using \eqref{eq:MultiplyConnectedHTAH}, it follows that the matrix
\begin{equation}
    \begin{bmatrix} \(\bmatL^\T \bmatL \)^{-1/4} \bmatL^\T \\ \mat H^\T \end{bmatrix}  \matTA^0  \begin{bmatrix} \bmatL \(\bmatL^\T \bmatL \)^{-1/4} & \mat H \end{bmatrix} \label{eq:TAExplicitLandH}
\end{equation}
is well-conditioned in $h$ up to the null space of the loop functions since the basis transformation matrix has full rank (up to the null space of the loop functions) and since the blocks on the main diagonal are well-conditioned and all blocks are bounded: the boundedness of $\mat H^\T \matTA^0 \bmatL (\bmatL^\T \bmatL )^{-1/4}$ and of $(\bmatL^\T \bmatL )^{-1/4} \bmatL^\T \matTA^0 \mat H$ follows from the boundedness of the blocks on the main diagonal.
We note that we cannot exclude that some singular values are shifted close to zero by the off-diagonal block matrices (though an $h$-refinement would not further alter these singular values). In practice, however, we did not observe such behavior.

Now, it remains to return from the explicit quasi-Helmholtz decomposition of \eqref{eq:TAExplicitLandH} to the new formulation in \eqref{eq:TAonMP}.
\begin{prop}\label{prop:WellconditionedVectorPotwithH}
    We have the spectral equivalence
    \begin{equation}
        \vec x^\T \Plh \(\matTA^k\)^\dagger \( \matL \mat G_{\lambda \lambda}^{-1} \matL^\T +  \Plh \) \matTA^k \Plh \vec x \asymp \vec x^\T \Plh \vec x\quad\text{for all~} \vec x \in \R^{N}\, . \label{eq:TAP2017:WellConditionedVectorPotwithH}
    \end{equation}
\end{prop}
\begin{proof}  
    We define 
    \begin{equation}
        \mat T_{\bmatL \mat H} \coloneqq  \begin{bmatrix} \bmatL^\T \\ \mat H^\T \end{bmatrix}  \matTA^0  \begin{bmatrix} \bmatL  & \mat H \end{bmatrix}\, ,
    \end{equation}
    \begin{equation}
        \mat Q_{\bmatL \mat H} \coloneqq  \begin{bmatrix} \(\bmatL^\T \bmatL \)^{-1/4} & \vecO  \\ \vecO & \matI \end{bmatrix}\, ,
    \end{equation}
    and observe
    \begin{equation}
        \begin{bmatrix} \(\bmatL^\T \bmatL \)^{-1/4} \bmatL^\T \\ \mat H^\T \end{bmatrix}  \matTA^0  \begin{bmatrix} \bmatL \(\bmatL^\T \bmatL \)^{-1/4} & \mat H \end{bmatrix} =  \mat Q_{\bmatL \mat H} \mat T_{\bmatL \mat H} \mat Q_{\bmatL \mat H}\, .
    \end{equation}
    In other words, we can interpret $\mat Q_{\bmatL \mat H}$ as a preconditioner for the standard loop/global loop discretized $ \mat T_{\bmatL \mat H} $. 
    
    We note that $ \mat T_{\bmatL \mat H} $ and $\mat Q_{\bmatL \mat H}$ are symmetric matrices and that they have the same null space (i.e., the null space due to the linear dependency of the loop functions). Summarizing, we have the Rayleigh quotient
    \begin{equation}
        \vec x^\T \mat Q_{\bmatL \mat H}  \mat T_{\bmatL \mat H}   \mat Q_{\bmatL \mat H}   \vec x \asymp  \vec x^\T \vec x \quad\text{for all~} \vec x \in \(\Null \mat Q_{\bmatL \mat H} \)^\perp\, .
    \end{equation}
    By using the substitution $\vec y =  \mat Q_{\bmatL \mat H} \vec x$, we obtain
    \begin{equation}
        \vec y^\T   \mat T_{\bmatL \mat H}  \vec y \asymp  \vec y^\T \mat Q_{\bmatL \mat H}^{-2}   \vec y \quad\text{for all~} \vec y \in \(\Null \mat Q_{\bmatL \mat H} \)^\perp
    \end{equation}
    from which immediately 
    \begin{equation}
        \vec y^\T   \mat T_{\bmatL \mat H}^2  \vec y \asymp  \vec y^\T \mat Q_{\bmatL \mat H}^{-4}   \vec y \quad\text{for all~} \vec y \in \(\Null \mat Q_{\bmatL \mat H} \)^\perp
    \end{equation}
    and thus
    \begin{equation}
        \vec x^\T   \mat Q_{\bmatL \mat H}^{2} \mat T_{\bmatL \mat H}^2  \mat Q_{\bmatL \mat H}^{2} \vec x \asymp  \vec x^\T   \vec x \quad\text{for all~} \vec x \in \(\Null \mat Q_{\bmatL \mat H} \)^\perp
    \end{equation}
    follows. 

    In addition, we observe that
    \begin{equation}
        \mat T_{\bmatL \mat H}^2 =   \begin{bmatrix} \bmatL^\T \\ \mat H^\T \end{bmatrix}  \matTA^0  \begin{bmatrix} \bmatL & \mat H \end{bmatrix}\begin{bmatrix} \bmatL^\T \\ \mat H^\T \end{bmatrix}  \matTA^0  \begin{bmatrix} \bmatL & \mat H \end{bmatrix}   = 
        \begin{bmatrix} \bmatL^\T \\ \mat H^\T \end{bmatrix}  \matTA^0  \(\bmatL \bmatL^\T + \mat H \mat H^\T \) \matTA^0  \begin{bmatrix} \bmatL & \mat H \end{bmatrix}
        \label{eq:TLHsquared}
    \end{equation}
    using $\mat H^\T \bmatL =\matO$.
    
    The global loop transformation matrix is not uniquely defined, but a possible transformation matrix can always be constructed from $\Ph \coloneqq \matI - \Pl - \Ps$ by using its \ac{SVD} so that $\mat H$ is the column space of it. Hence, we can always obtain
    \begin{equation}
        \Ph = \mat H \mat H^\T \, . \label{eq:ProjectorHH}
    \end{equation}

    By using \eqref{eq:TLHsquared} and \eqref{eq:ProjectorHH}, we yield
    \begin{equation}
        \mat Q_{\bmatL \mat H}^{2} \mat T_{\bmatL \mat H}^2  \mat Q_{\bmatL \mat H}^{2}
        =  \begin{bmatrix}      \(\bmatL^\T \bmatL \)^{-1/2}  \bmatL^\T \\ \mat H^\T \end{bmatrix}  \matTA^0  \(\bmatL \bmatL^\T + \Ph \) \matTA^0  \begin{bmatrix} \bmatL  \(\bmatL^\T \bmatL \)^{-1/2}  & \mat H \end{bmatrix} \, .
    \end{equation}
    
    We also note that the transformation $\begin{bmatrix} \bmatL  (\bmatL^\T \bmatL )^{-1/2}  & \mat H \end{bmatrix} $ is well-conditioned, in fact,
    \begin{equation}
        \begin{bmatrix} \bmatL  \(\bmatL^\T \bmatL \)^{-1/2}  & \mat H \end{bmatrix}  \begin{bmatrix}   \(\bmatL^\T \bmatL \)^{-1/2}  \bmatL^\T \\ \mat H^\T \end{bmatrix} =  \Plh \, .
    \end{equation}
    Thus we have
    \begin{equation}
        \vec x^\T \begin{bmatrix} \bmatL  \(\bmatL^\T \bmatL \)^{-1/2}  & \mat H \end{bmatrix} \mat Q_{\bmatL \mat H}^{2} \mat T_{\bmatL \mat H}^2  \mat Q_{\bmatL \mat H}^{2} \begin{bmatrix}   \(\bmatL^\T \bmatL \)^{-1/2}  \bmatL^\T \\ \mat H^\T \end{bmatrix}  \vec x^\T 
        \asymp \vec x^\T \Plh \vec x \quad\text{for all~} \vec x \in \R^N\, ,
    \end{equation}
    where we note that the preconditioned system matrix can be expressed as
    \begin{equation}
        \begin{bmatrix} \bmatL  \(\bmatL^\T \bmatL \)^{-1/2}  & \mat H \end{bmatrix} \mat Q_{\bmatL \mat H}^{2} \mat T_{\bmatL \mat H}^2  \mat Q_{\bmatL \mat H}^{2} \begin{bmatrix}   \(\bmatL^\T \bmatL \)^{-1/2}  \bmatL^\T \\ \mat H^\T \end{bmatrix} 
        =
        \Plh \matTA^0  \(\bmatL \bmatL^\T + \Ph \) \matTA^0  \Plh\, .
    \end{equation}
    We can replace $\Ph$ by $\Plh$ since the matrix $\Plh \matTA^0  \Pl \matTA^0  \Plh$ is symmetric, positive definite and
    \begin{equation}
        \norm{\Plh \matTA^0  \Pl \matTA^0  \Plh}_2 \lesssim  \norm{\Plh}_2 \norm{\matTA^0}_2 \norm{\Pl}_2 \norm{\matTA^0}_2 \norm{\Plh}_2 \lesssim 1
    \end{equation}
    is bounded, where $\norm{\matTA^0}_2 \lesssim 1$ follows from the compactness of $\TA$.
    Likewise, the dynamic kernel is a compact perturbation and by substituting back from $\bmatL$ to $\matL$ and $\mat G_{\lambda \lambda}^{-1}$, we obtain that the matrix
    \begin{equation}
        \Plh \(\matTA^k\)^\dagger \( \matL \mat G_{\lambda \lambda}^{-1} \matL^\T +  \Plh \) \matTA^k \Plh \vec x \asymp \vec x^\T \Plh \vec x \quad\text{for all~} \vec x \in \R^N
    \end{equation}
    is well-conditioned (up to its null space).
\end{proof}

\subsection{Scalar Potential}\label{sec:ScalarPotential}
In this section, we are going to establish that the matrix
\begin{equation}
    \Pgs^\T \(\matTPhi^k\)^\dagger \bPmS \matTPhi^k \Pgs 
\end{equation}
is well-conditioned (up to its null space).

As for the vector potential operator, we need some lemmas and auxiliary matrices.
We define the matrix
\begin{equation}
    \omat V  \coloneqq \(\matS^\T \matS \)^+ \matS^\T \matS \mat V \matS^\T \matS \(\matS^\T \matS \)^+\, .
\end{equation}
This matrix is important since it is connected to the scalar potential by (see \eqref{eq:TPhiandV})
\begin{equation}
    \omat V \equiv \(\matS^\T \matS \)^+ \matS^\T  \matTPhi^0 \matS \(\matS^\T \matS \)^+\, . \label{eq:omatVtoScalarPot}
\end{equation}
\begin{lem} \label{lem:VequivVhat}
    We have the spectral equivalence
    \begin{equation}
        \vec x^\T  \mat V \mat G_{pp}^{-1} \mat V \vec x \asymp \vec x^\T \omat V  \mat G_{pp}^{-1} \omat V \vec x \quad\text{for all~} \vec x \in   \(\Span \oS\)^\perp \, . \label{eq:VEquivVhat}
    \end{equation}
\end{lem}
\begin{proof}
    If $\vec x$ is such that  $\oS^\T \vec x = 0$, then we have
    \begin{equation}
        \matS^\T \matS \(\matS^\T \matS \)^+ \vec x = \vec x
    \end{equation} 
    since for a symmetric, positive semi-definite matrix the null space is orthogonal to the column and row space,
    and thus
    \begin{equation}
        \vec x^\T \mat V \vec x  \asymp \vec x^\T \omat V \vec x \, .
    \end{equation}
    Clearly, we have for such $\vec x$ also
    \begin{equation}
        \vec x^\T \mat V \mat V \vec x  = \vec x^\T \omat V \omat V \vec x \, ,
    \end{equation}
    and since $ \vec y^\T \mat G_{pp}^{-1} \vec y \asymp \vec y^\T \vec y$ holds for all $\vec y\in \R^{N_\mr{C}}$, we have
    \begin{equation}
        \vec x^\T \mat V \mat G_{pp}^{-1} \mat V \vec x  = \vec x^\T\omat V \mat G_{pp}^{-1} \omat V \vec x\quad\text{for all~} \vec x \in (\oS)^\perp  \, .
    \end{equation}
\end{proof}

\begin{cor}\label{lem:CLaplacianGraph_Sigma}
    We have the spectral equivalence
    \begin{equation}
        \vec x^\T \cwmat \Delta  \vec x \asymp \vec x^\T \( \matS^\T \matS + \oS \oS^\T h^4 \) \vec x\quad\text{for all~} \vec  x \in \R^{N_\mr{C}}\, .
    \end{equation}
\end{cor}
\begin{proof}
    Follows from \Cref{lem:CLaplacianGraph}.
\end{proof}

\begin{lem} \label{lem:mixedGramEig}
    The vector $\oS $ is a right eigenvector of $\mat G_{\wt\lambda p}^{-\T}  $, that is, $ \oS = \mat G_{\wt\lambda p}^{-\T}  \oS$.
\end{lem}
\begin{proof}
    If
    \begin{equation}
        \wt \lambda_{\vec y} = p_{\vec x} = 1\quad\text{for all~} \veg r \in \Gamma\, , \label{eq:LandPEo}
    \end{equation}  
    then $\vec y = \oS$ and $\vecel{\vec x}_i =  A_i$, where $\wt \lambda_{\vec y} = \sum_{n=1}^{N_\mr{V} } \vecel{\vec y}_n \wt \lambda_n $ and $p_{\vec x} = \sum_{n=1}^{N_\mr{C}}  \vecel{\vec x}_n p_n$. Testing \eqref{eq:LandPEo} with $p_i$ yields
    \begin{equation}
        \mat G_{p\wt \lambda} \oS = \mat G_{pp} \vec x = \oS\, .
    \end{equation}
    Since $\mat G_{p\wt \lambda} = \mat G_{\wt \lambda p}^\T$, we have $ \oS = \mat G_{\wt \lambda p}^{-\T}\oS$.
\end{proof}

\begin{cor}
    For any mean value free vector $\vec x$, that is, $\oS^\T \vec x = 0$, we have that the vector $\mat G_{\wt \lambda p}^{-1} \vec x$ is mean value free as well.
\end{cor}

\begin{proof}
    This follows from \Cref{lem:mixedGramEig} since if we have $\oS^\T \vec x = 0$, then
    \begin{equation}
        \oS^\T\mat G_{\wt\lambda p}^{-1} \vec x =  \oS^\T   \vec x = 0\, .
    \end{equation}
\end{proof}

\begin{prop} \label{prop:WellCondTPhi}
    We have the spectral equivalence
    \begin{equation}
        \vec x^\T \Pgs^\T \(\matTPhi^k\)^\dagger \bPmS \matTPhi^k \Pgs \vec x  \asymp \vec x^\T \Ps \vec x\quad\text{for all~} \vec x \in \R^{N}\, .\label{eq:WellConditionedScalarPot}
    \end{equation}
\end{prop}
\begin{proof}
    We start with \eqref{eq:WWdualDelta}, that is
    \begin{equation}
        \vec x^\T \htmat{W} \mat G_{\wt\lambda \wt\lambda}^{-1}  \htmat{W} \vec x \asymp \vec x^\T \htmat{\Delta} \vec x\quad\text{for all~} \vec x \in  \R^{N_\mr{C}} \label{eq:WGWDeltaDual}
    \end{equation}
    and applying the substitution $\vec x = \mat G_{\wt \lambda \wt \lambda}^{-1/2} \vec y$ yields
    \begin{equation}
        \vec y^\T  \mat G_{\wt \lambda \wt \lambda}^{-1/2} \htmat{W} \mat G_{\wt\lambda \wt\lambda}^{-1}\htmat{W} \mat G_{\wt\lambda \wt\lambda}^{-1/2} \vec y 
        = \vec y^\T \( \mat G_{\wt \lambda \wt \lambda}^{-1/2} \htmat{W} \mat G_{\wt\lambda \wt\lambda}^{-1/2}\)^2 \vec y  
        \asymp \vec y^\T  \mat G_{\wt \lambda \wt \lambda}^{-1/2} \htmat \Delta  \mat G_{\wt \lambda \wt \lambda}^{-1/2} \vec y\quad\text{for all~} \vec y \in \R^{N_\mr{C}}\, ,\label{eq:wWwWwLapL}
    \end{equation}
    and hence
    \begin{equation}
        \vec y^\T  \mat G_{\wt \lambda \wt \lambda}^{1/2} \htmat{W}^{-1} \mat G_{\wt\lambda \wt\lambda}^{1/2} \vec y \asymp \vec y^\T   \(\mat G_{\wt \lambda \wt \lambda}^{-1/2} \htmat \Delta  \mat G_{\wt \lambda \wt \lambda}^{-1/2} \)^{-1/2}\vec y\quad\text{for all~} \vec y \in \R^{N_\mr{C}} \, .\label{eq:DualWLapL}
    \end{equation}
    From the Calderón identities and the theory outlined in \cite{steinbach_construction_1998, buffa_dual_2007, hiptmair_operator_2006}, we have 
    \begin{equation}
        \vec x^\T \htmat W \vec x \asymp  \vec x^\T \mat G_{\wt  \lambda p}^{-\T}  \mat V^{-1} \mat G_{\wt  \lambda p}^{-1} \vec x \quad\text{for all~} \vec x\in \R^{N_\mr{C}} \, .
    \end{equation}
    Inserting this in \eqref{eq:DualWLapL}, applying the back-substitution $\vec y = \mat G_{\wt \lambda \wt \lambda}^{1/2} \vec x$,  squaring the matrices on both sides and inverting them, we obtain
    \begin{equation}
        \vec x^\T \mat G_{\wt\lambda p}^{-\T}  \mat V \mat G_{\wt\lambda p}^{-1} \mat G_{\wt \lambda \wt \lambda} \mat G_{\wt\lambda p}^{-\T}  \mat V \mat G_{\wt\lambda p}^{-1}  \vec x  \asymp \vec x^\T \htmat \Delta^{-1} \vec x \quad\text{for all~} \vec x \in \R^{N_\mr{C}}\, . \label{eq:HerewesimplifyRHS}
    \end{equation}
    The right-hand side can be simplified:
    it was shown in \cite{andriulli_wellconditioned_2013a} that
    \begin{equation}
        \( \matS^\T \matS + \oS \oS^\T / N_\mr{C} \)^{-1} =  \( \matS^\T \matS\)^{+} +  \oS \oS^\T  /N_\mr{C}
    \end{equation}
    holds.
    In addition with $N_\mr{C} \asymp 1/h^2$ and \Cref{lem:CLaplacianGraph_Sigma}, we can simplify the right-hand side of \eqref{eq:HerewesimplifyRHS}  yielding
    \begin{equation}
        \vec x^\T \mat G_{\wt\lambda p}^{-\T}  \mat V \mat G_{\wt\lambda p}^{-1} \mat G_{\wt \lambda \wt \lambda} \mat G_{\wt\lambda p}^{-\T}  \mat V \mat G_{\wt\lambda p}^{-1}  \vec x  \asymp \vec x^\T \(  \( \matS^\T \matS\)^{+} +  \oS \oS^\T \) \vec x \quad\text{for all~} \vec x \in \R^{N_\mr{C}}\, . \label{eq:BeforeGduallamdagoes}
    \end{equation}
    From \cite{steinbach_construction_1998, buffa_dual_2007}, we can obtain 
    \begin{equation}
        \vec x^\T \mat G_{pp}^{-1} \vec x \asymp   \vec x^\T   \mat G_{\wt\lambda p}^{-1} \mat G_{\wt \lambda \wt \lambda} \mat G_{\wt\lambda p}^{-\T}\vec x \quad\text{for all~} \vec x \in \R^{N_\mr{C}}\, .
    \end{equation}
    Inserting this into \eqref{eq:BeforeGduallamdagoes} yields
    \begin{equation}
        \vec x^\T \mat G_{\wt\lambda p}^{-\T} \mat V \mat G_{pp}^{-1} \mat V \mat G_{\wt\lambda p}^{-1}  \vec x   \asymp \vec x^\T \(  \( \matS^\T \matS\)^{+} +  \oS \oS^\T \) \vec x \quad\text{for all~} \vec x \in \R^{N_\mr{C}}\, .
    \end{equation}
    Then we use the substitution $\vec x = \matS^\T \vec y$ and obtain
    \begin{equation}
        \vec y^\T  \matS \mat G_{\wt\lambda p}^{-\T}  \mat V \mat G_{pp}^{-1} \mat V \mat G_{\wt\lambda p}^{-1} \matS^\T \vec y \\ \asymp \vec y^\T \Ps \vec y \quad\text{for all~} \vec y \in \R^{N}\, ,
    \end{equation}
    since  $\matS \oS = \vecO$. Due to this relationship, it is clear that all vectors $\matS^\T \vec y$ with $\vec y \in \R^N$ are mean value free, that is, we have  $\oS^\T \matS^\T \vec y = 0$. Thus we can invoke \Cref{lem:VequivVhat} and obtain
    \begin{equation}
        \vec y^\T  \matS \mat G_{\wt\lambda p}^{-\T}  \omat V \mat G_{pp}^{-1} \omat V \mat G_{\wt\lambda p}^{-1} \matS^\T \vec y \\ \asymp \vec y^\T \Ps \vec y \quad\text{for all~} \vec y \in \R^{N}\, . \label{eq:TPhiComparison}
    \end{equation}
    Inserting the right-hand side from \eqref{eq:omatVtoScalarPot}, we obtain
    \begin{equation} 
        \vec y^\T \Pgs^\T \matTPhi^0 \bPmS \matTPhi^0 \Pgs \vec y  \asymp \vec y^\T \Ps \vec y \quad\text{for all~} \vec y \in \R^{N}\, , 
    \end{equation} 
    where $\Pgs$ has been defined in \eqref{eq:Pgs}.
    As in \Cref{prop:WellconditionedVectorPot}, we note that the dynamic kernel only introduces a compact perturbation, and that by using $(\matTPhi^k)^\dagger$ for the left scalar potential operator matrix in \eqref{eq:WellConditionedScalarPot}, we yield a symmetric, positive semi-definite system.
\end{proof}

\subsection{Preconditioned Electric Field Integral Equation}\label{sec:EFIEWell}
\begin{prop}
    The new formulation is well-conditioned in the static limit, that is, the matrix in \eqref{eq:NewFormulation} satisfies
    \begin{equation}
        \lim_{k \rightarrow 0} \vec x^\T \bPo^\dagger \mat T^\dagger \bPm \mat T \bPo  \vec x \asymp \vec x^\T \vec x\quad\text{for all~} \vec x \in \R^N\, . \label{eq:FinalInequalityNewCalderon}
    \end{equation}
\end{prop}
\begin{proof}
    We recall the result from \eqref{eq:LastProofBigUmformung}, where we found
    \begin{equation}
        \lim_{k \rightarrow 0}     \bPo^\dagger \mat T^\dagger \bPm \mat T \bPo  =  \Plh^\dagger \(\matTA^0\)^\dagger \PmL \matTA^0 \Plh + \Pgs^\dagger \(\matTPhi^0\)^\dagger \bPmS \matTPhi^0 \Pgs\, .
    \end{equation}
    Clearly, the new formulation is low-frequency stable and the well-conditionedness in $h$ follows from the orthogonality of the null spaces of $\Plh$ and $\Pgs$,  \Cref{prop:WellconditionedVectorPotwithH}, and \Cref{prop:WellCondTPhi} so that we have
    \begin{equation}
        \vec x^\T  \( \Plh^\dagger \(\matTA^0\)^\dagger \PmL \matTA^0 \Plh + \Pgs^\dagger \(\matTPhi^0\)^\dagger \bPmS \matTPhi^0 \Pgs\) \vec x \asymp  \vec x^\T \vec x \quad\text{for all~} \vec x \in \R^N\, . \label{eq:AdditionOfPrincipalParts}
    \end{equation}
\end{proof}
For the dynamic case, we note that the additional terms appearing in \eqref{eq:LastProofBigUmformung} have at least up to a certain frequency a smaller norm than the principal terms in \eqref{eq:AdditionOfPrincipalParts}. Numerical evidence suggests that the frequency starts to impact the iterative solver when the wavelength is the range of, or even larger than, the electric size of the geometry.

\subsection{Analysis of the Computational Complexity}
An important aspect for any preconditioner is its computational complexity (i.e., how time and memory consumption scale when $N$ is increased) and it is important to ensure that the complexity improves compared with the original formulation. When a fast method such as the \ac{MLFMM} or the \ac{ACA} is used, then the cost for obtaining $\mat T$ and the memory consumption scale as $\mc O(N \log N)$ in the admissible frequency range (i.e., admissible with respect to the employed fast method). The cost for obtaining the Gram matrices and the loop and star transformation matrices is $\mc O(N)$, since the number of neighboring cells or vertices is independent of $N$. The matrix $\mat T^\dagger$ does not need to be computed explicitly. In fact, we note that 
\begin{equation}
    \mat T^\dagger \vec x = \overline{\overline{\mat T^\dagger \vec x } } = \overline{\overline{\overline{\mat T}^\T \vec x } } = \overline{\mat T^\T \overline{\vec x } } = \overline{\mat T \overline{\vec x }}\, .
\end{equation}

If an iterative solver is used, the complexity to obtain a solution is $\mc O(N_\text{iter} N \log N)$: it is the product of the total number of iterations and the cost of a single matrix-vector product. Since all the appearing Gram matrices  have $\mc O(N)$ non-zero elements, the complexity of a single matrix-vector product is $\mc O(N)$. As the Gram matrices are well-conditioned in $h$, the cost of a single matrix-vector product involving an inverse Gram matrix is $\mc O(N)$ as well. For the projectors, graph Laplacian systems must be solved. For this task, algorithms with an observed complexity of $\mc O(N)$ (e.g., algebraic multigrid techniques such as \cite{livne_lean_2012}) or with a guaranteed complexity of $\mc O\(N  \log^{\mc O(1)}N\)$ (e.g., the simple, combinatorial solver presented in \cite{kelner_simple_2013}) are available. 

We have shown in this contribution, that the condition number of the preconditioned system matrix is bounded independent from $h$. If the conjugate gradient solver is used to solve for $\vec x$ in a linear system of equations $\mat A \vec x = \vec b$, where $\mat A$ is an \ac{HPD} matrix, then the number of iterations is bounded by 
\begin{equation} 
    N_\text{iter} \leq \left\lceil (1/2) \sqrt{\cond\( \mat A\)} \log\(2/\varepsilon \) \right\rceil \label{eq:Niter}
\end{equation}
where $\varepsilon$ is the solver tolerance. Thus if the \ac{CG} method is used to solve \eqref{eq:NewFormulation}, the number of iterations is bounded independent of $h$.  

If we assume that the costs for solving the graph Laplacian systems are $\mc O(N)$, then the complexity of a single iteration is $\mc O(N \log N)$ and thus together with \eqref{eq:Niter} the total costs for obtaining a solution is $\mc O(N \log N)$. We note that if we use \eqref{eq:Niter} also in the case of the unpreconditioned $\mat T$, then the complexity is $\mc O(N^{1.5} \log N)$ (note that this bound is actually not applicable since the conjugate gradient method cannot be used for non-Hermitian and indefinite matrices such as $\mat T$, but other Krylov methods do not have, to the best of the authors' knowledge, a sharper bound).

\section{Numerical Results}\label{sec:NumericalResults}
For the implementation of the preconditioner, we did not use the wavenumber $k$ directly to cure the low-frequency breakdown. Instead, we used the following definitions
\begin{align}
    \Po &\coloneq  \Plh / \alpha + \im \Ps / \beta\, , \\
    \PmL &\coloneq  \matL \mat G_{\lambda \lambda}^{-1} \matL^\T/\alpha^2 + \Plh/\gamma\, , \\
    \PmS &\coloneq \matS \(\matS^\T \matS\)^+ \mat G_{pp}^{-1} \(\matS^\T \matS\)^+  \matS^\T /\beta^2 \, ,             
\end{align}    
where 
\begin{align}
    \alpha &= \sqrt[4]{\norm{ \Plh \matTA^\dagger \matL \mat G_{\lambda \lambda}^{-1} \matL^\T\matTA \Plh}_2}\, , \\
    \beta &= \sqrt[4]{\norm{\Ps \matTPhi^\dagger \PmS \matTPhi \Ps}_2}\, , \\
    \gamma  &= \norm{(\Plh/\alpha) \matTA^\dagger \Plh \matTA (\Plh/\alpha) }_2\, .
\end{align}
These norms are estimated using the power iteration algorithm. Typically the condition number obtained by using norms is lower than using
\begin{align}
    \alpha &= \sqrt{k}\, , \\
    \beta &= 1/\sqrt{k}\, , \\
    \gamma  &= k\, ;
\end{align}
thereby, the number of iterations used by a Krylov subspace method is reduced (and this saving usually outweighs the costs for estimating the norms).

First, we considered a sphere, radius \SI{1}{\meter}, to confirm the low-frequency stability by computing the condition number obtained by the new formulation and compared it with a loop-tree preconditioned system. \Cref{fig:SphereF} shows that the new formulation is frequency stable and \Cref{fig:Sphere_RCS} that the bistatic radar cross section can be accurately computed down to \SI{1e-25}{\hertz}. The saturation of the condition number in the case ``no preconditioner'' stems from numerical cancellation: the null space of $\matTPhi$ exists only up to numerical precision and when $k$ becomes too small, the (numerical) norm of the null space of $\matTPhi$ is larger than the norm  of $\matTA$ so that $\matTA$ completely vanishes in numerical noise. 
To verify the dense-discretization stability, we computed the condition number for the new formulation and the loop-tree preconditioned system for an increasing spectral index $1/h$. We can see from \Cref{fig:SphereH} that the new formulation is dense-discretization stable, whereas the loop-tree preconditioner is not.

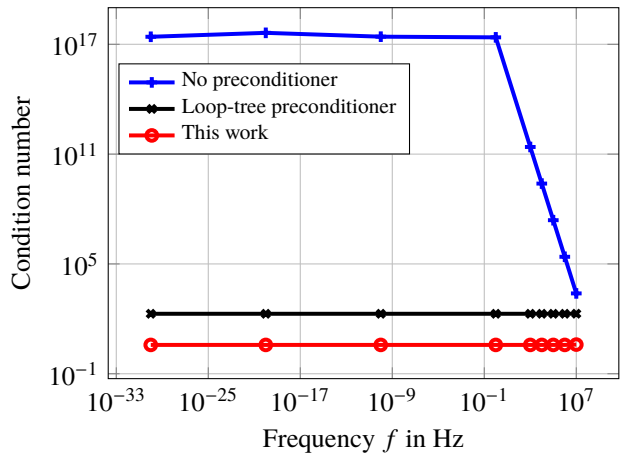
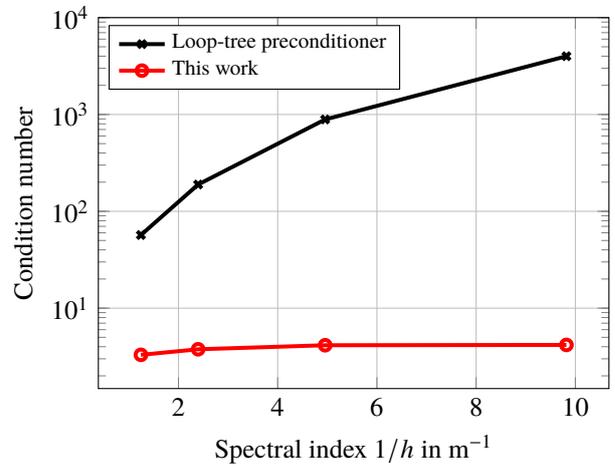
\begin{figure}
    \begin{subfigure}[c]{0.49\textwidth}
        \centering
        \begin{tikzpicture}
\begin{axis}[%
width=8.3cm,
height=6.5cm,
xlabel={Spectral index $1/h$  in \si{\per\meter}},
xmajorgrids,
xmode=log,
ymode=log,
xlabel={Frequency $f$ in Hz},
ylabel={Condition number},
ymax=10000000000000000000,
ymajorgrids,
legend style={ anchor=north west,
    at={(0.02,0.85)}
}
]

\pgfplotstableread[col sep=semicolon,  header=true]{./Results/Elsevier_Sphere_Freq_sindex.csv}\MIGrcsfielddatatable

\addplot [color=blue,  mark=+]
table[x=f, y=0]   from \MIGrcsfielddatatable {};
\addlegendentry{No preconditioner};

\addplot [color=black,  mark=x]
table[x=f, y=100]   from \MIGrcsfielddatatable {};
\addlegendentry{Loop-tree preconditioner};

\addplot [color=red , mark=o]
table[x=f, y=30]   from \MIGrcsfielddatatable {};
\addlegendentry{This work};


\end{axis}
\end{tikzpicture}%
        \subcaption{The condition number as a function of the frequency. The spectral index is $\SI{2.4}{\per\meter}$.}
        \label{fig:SphereF}
    \end{subfigure}
    \quad
    \begin{subfigure}[c]{0.49\textwidth}
        \centering
        \begin{tikzpicture}
\begin{axis}[%
width=8.3cm,
height=6.5cm,
xlabel={Spectral index $1/h$  in \si{\per\meter}},
xmajorgrids,
ymode=log,
ylabel={Condition number},
ymax=10000,
ymajorgrids,
legend style={ anchor=north west,
    at={(0.02,0.98)}
}
]

\pgfplotstableread[col sep=semicolon,  header=true]{./Results/New_Sphere_Cond_sindex.csv}\MIGrcsfielddatatable

\addplot [color=black,  mark=x]
table[x=sindex, y=100]   from \MIGrcsfielddatatable {};
\addlegendentry{Loop-tree preconditioner};

\addplot [color=red , mark=o]
table[x=sindex, y=30]   from \MIGrcsfielddatatable {};
\addlegendentry{This work};


\end{axis}
\end{tikzpicture}%
        \subcaption{The condition number as a function of the spectral index. The frequency is $\SI{1}{\mega\hertz}$.}
        \label{fig:SphereH} 
    \end{subfigure}
    \caption{Sphere: spectral analysis.}
\end{figure}

\begin{figure}
    \begin{subfigure}[c]{0.49\textwidth}
        \centering
        \begin{tikzpicture}
\begin{axis}[%
width=8.3cm,
height=6.5cm,
xlabel={Spectral index $1/h$  in \si{\per\meter}},
xmajorgrids,
xticklabel={\pgfmathprintnumber[precision=1]{\tick}\si{\degree}},
xlabel={Angle $\theta$},
ylabel={RCS in dBsm},
ymajorgrids,
legend style={ anchor=south west,
    at={(0.02,0.02)}
}
]

\pgfplotstableread[col sep=semicolon,  header=true]{./Results/NewCalderon_RCS_SphereR1h01_F1.0e6.csv}\MIGrcsfielddatatable

\addplot [color=black]
table[x=theta, y=eabs_rcs_mie]   from \MIGrcsfielddatatable {};
\addlegendentry{Mie series};

\addplot [color=red , only marks, mark=o]
table[x=theta, y=eabs_rcs_31]   from \MIGrcsfielddatatable {};
\addlegendentry{This work};


\end{axis}
\end{tikzpicture}%
        \subcaption{Frequency: \SI{1}{\mega\hertz}.}
        \label{fig:FarField1M}
    \end{subfigure}
    \quad
    \begin{subfigure}[c]{0.49\textwidth}
        \centering
        \begin{tikzpicture}
\begin{axis}[%
width=8.3cm,
height=6.5cm,
xlabel={Spectral index $1/h$  in \si{\per\meter}},
xmajorgrids,
xticklabel={\pgfmathprintnumber[precision=1]{\tick}\si{\degree}},
xlabel={Angle $\theta$},
ylabel={RCS in dBsm},
ymajorgrids,
legend style={ anchor=south west,
    at={(0.02,0.02)}
}
]

\pgfplotstableread[col sep=semicolon,  header=true]{./Results/NewCalderon_RCS_SphereR1h01_F1.0e-25.csv}\MIGrcsfielddatatable

\addplot [color=black]
table[x=theta, y=eabs_rcs_mie]   from \MIGrcsfielddatatable {};
\addlegendentry{Mie series};

\addplot [color=red , only marks, mark=o]
table[x=theta, y=eabs_rcs_30]   from \MIGrcsfielddatatable {};
\addlegendentry{This work};


\end{axis}
\end{tikzpicture}%
        \subcaption{Frequency: \SI{1e-25}{\hertz}.}
        \label{fig:FarFieldM30Hz} 
    \end{subfigure}
    \caption{Sphere: bistatic radar cross section.}
    \label{fig:Sphere_RCS} 
\end{figure}
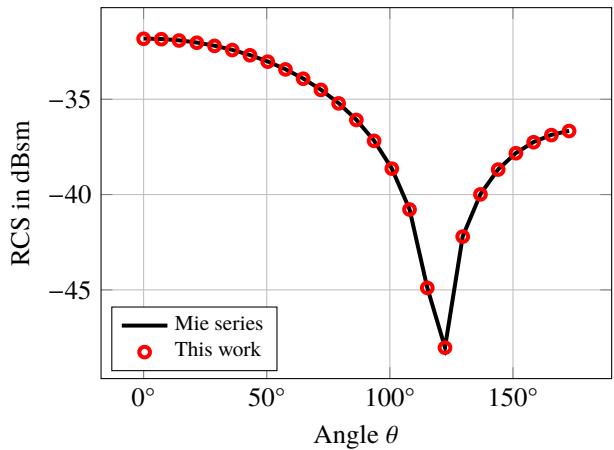
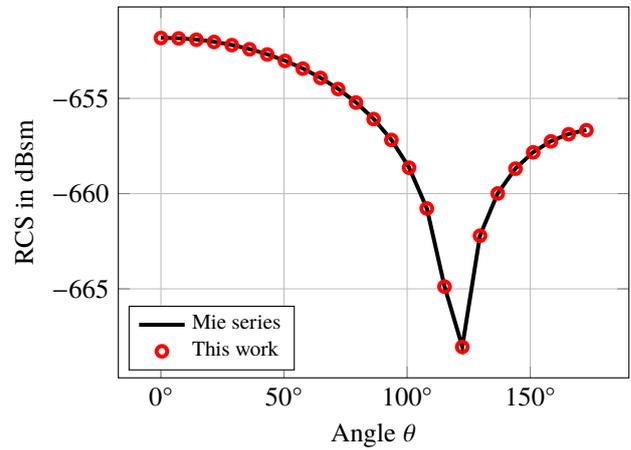

In order to verify that the proposed method works also for topologically non-trivial geometries, we used a simple, multiply connected structure, where we refined the mesh structuredly. As for the sphere, we compared the proposed preconditioner against a loop-tree preconditioner. \Cref{fig:CubeTorusH} and \Cref{fig:CubeTorusF} show that the proposed scheme remains stable for multiply connected geometries.

\begin{figure}
    \begin{subfigure}[c]{0.49\textwidth}
        \centering
        \input{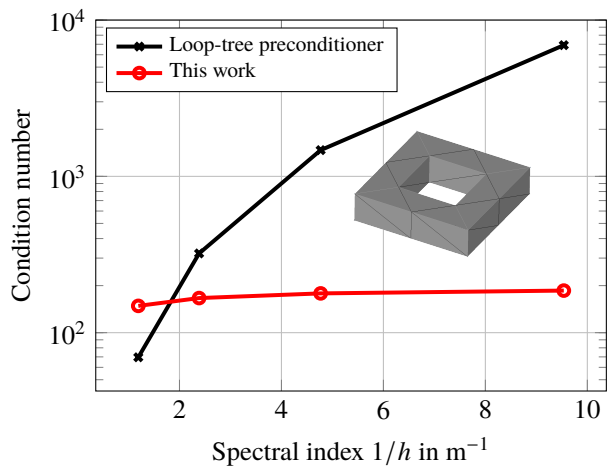}
        \caption{Toroidal structure: the condition number as a function of the spectral index. The frequency is $\SI{1}{\mega\hertz}$.}
        \label{fig:CubeTorusH} 
    \end{subfigure}
    \quad
    \begin{subfigure}[c]{0.49\textwidth}
        \centering
        \input{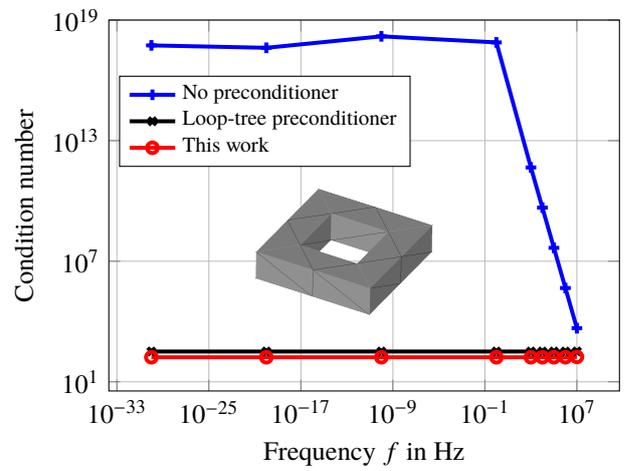}
        \caption{Toroidal structure: the condition number as a function of the frequency. The spectral index is $\SI{2.4}{\per\meter}$.}
        \label{fig:CubeTorusF} 
    \end{subfigure}
    \caption{Spectral analysis of an open and a topologically non-trivial structure.}
\end{figure}

Finally, we considered a more realistic structure. To compress the system matrix, we used an \ac{ACA} with tolerance \num{1e-4}. Other fast methods could be used as well as long as they are admissible in the respective frequency range. For typical frequency ranges no modification of these fast methods or any other part of the original code is necessary. For extremely low frequencies, however, it is necessary to avoid numerical cancellation of the vector potential due to the scalar potential. This can be either achieved by storing and compressing the vector potential separate from the scalar potential \cite{echeverribautista_hierarchical_2014}, or---in order to avoid these unnecessary numerical costs in construction time and memory consumption---by using a more refined fast scheme such as the one presented in \cite{andriulli_helmholtzstable_2012}. As an iterative solver, we used the \ac{CG} method for the new formulation and the \ac{CGS} method for the other formulations since the \ac{CG} method is only applicable if the matrix is \ac{HPD}. We note that a single iteration step of \ac{CGS} requires two matrix-vector products. We employed the AGMG library  \cite{notay_agmg_, napov_algebraic_2012} for the fast inversion of the graph Laplacians with solver tolerance \num{1e-14} to demonstrate that even for extreme small tolerances our preconditioner remains efficient. We have chosen a model of the Fokker Dr.I depicted in \Cref{fig:Fokker_Current}, which has \num{390} global loops. As excitation we considered both a plane wave and a voltage gap excitation. The model is discretized with \num{294420} unknowns resulting in a non-uniform mesh with $\cond \mat G_{\veg f \veg f} \approx \num{3e3} $ and $\cond \mat G_{pp} \approx \num{7e4}$. Its electric length is $\num{1e-3}\lambda$. From \Cref{tab:Fokker}, we can see that there is not only a significant reduction of the number of iterations, but also of time. 

\begin{figure}
    \centering
    \includegraphics[scale=0.2]{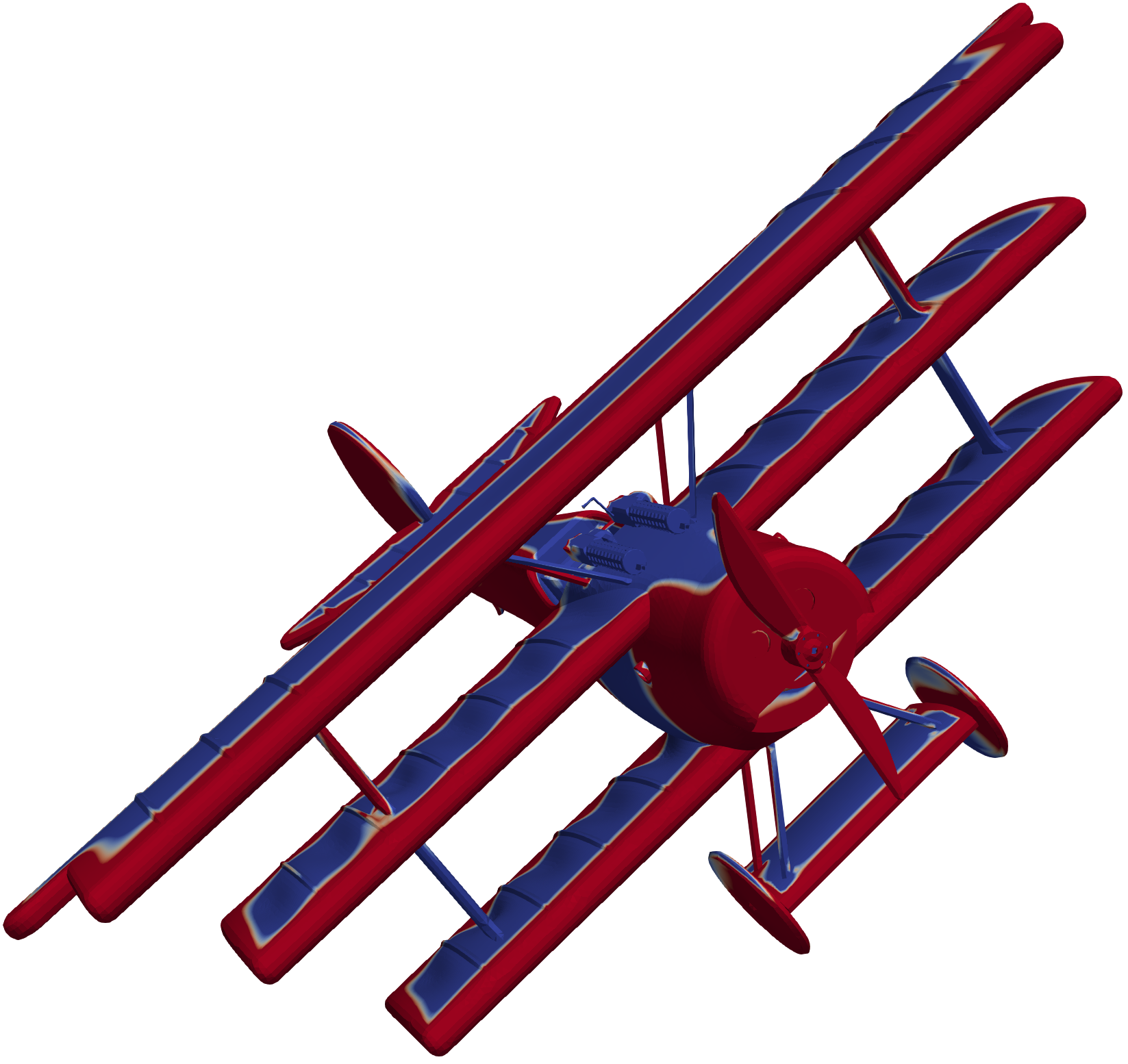}
    \caption{Fokker Dr.I: real part of $\veg j$ excited by an incident plane wave.}
    \label{fig:Fokker_Current}
\end{figure}

\begin{table}
    \centering
    \begin{minipage}{\linewidth}\centering
        \renewcommand{\footnoterule} {}
        \renewcommand{\thempfootnote}{\fnsymbol{mpfootnote}}
        \begin{tabular}{@{}*{2}{L S L S S}@{}}
            \toprule 
            \multicolumn{1}{N}{\small Preconditioner} &%
            \multicolumn{1}{N}{\small Iterations} &
            \multicolumn{1}{N}{\small Time\mpfootnotemark[1]} &
            \multicolumn{2}{N}{\small $l^2$-relative error\mpfootnotemark[2]} 
            \\
            \cmidrule(l){4-5}
            \multicolumn{3}{N}{} &
            \multicolumn{1}{N}{\small Current} &
            \multicolumn{1}{N}{\small RCS } 
            \\
            \multicolumn{2}{N}{} &
            \multicolumn{1}{N}{\small (h:m:s)} &
            \multicolumn{1}{N}{\small (\%)} &
            \multicolumn{1}{N}{\small (\%)} 
            \\
            \cmidrule(r){1-1} \cmidrule(lr){2-2}  \cmidrule(lr){3-3}  \cmidrule(lr){4-4}  \cmidrule(l){5-5}
            \multicolumn{5}{@{}L}{\emph{Plane wave excitation}} \\
            \small Loop-tree &  \small  4846  & \formatedtime{31}{46}{57} &     \\
            \small This Work&  \small 201 & \formatedtime{05}{15}{25} & \small  0.62721 &  \small 0.0007 \\
            \multicolumn{5}{@{}L}{\emph{Voltage gap excitation}} \\
            \small Loop-tree &  \small  2309  & \formatedtime{15}{35}{07} &     \\
            \small This Work&  \small 26 & \formatedtime{00}{55}{08} & \small  0.0158 &  \small 1.4638 \\
            \bottomrule \addlinespace
        \end{tabular}
        \footnotetext[1]{This is the total time including the setup time for the preconditioners.}
        \footnotetext[2]{The relative error is with respect to the solution obtained by using the loop-tree preconditioner.}
    \end{minipage}
    \caption{Fokker Dr.I: the number of iterations and the time used by the solver to obtain a residual error below \num{1e-4}.}
    \label{tab:Fokker}
\end{table}

\section{Conclusion}
We presented a preconditioner for the \ac{EFIE} that yields a Hermitian, positive definite, and well-conditioned system matrix. Due to the applicability of the \ac{CG} method, there is a---at least theoretically---guaranteed convergence and the complexity for obtaining such a solution is $\mc O(N \log N)$ if a fast method is used to compress the \ac{EFIE} system matrix (and if it is assumed that the algebraic multigrid methods allow to solve Laplacian systems in $\mc O(N)$ complexity). Preliminary results indicate that an extension to the \ac{CFIE} is possible \cite{adrian_wellconditioned_2016}.

\section*{Acknowledgement}
This work has been funded in part by the European Research Council (ERC) 
under the European Union's Horizon 2020 research and innovation program 
(ERC project 321, grant No. 724846).

\appendix 
\section{Proof of Spectral Equivalence of the Discretized Hypersingular and Laplace-Beltrami Operator}
\label{sec:Appendix}
In the following, we give a proof for the spectral equivalence of $\hmat W$ and $\hmat \Delta$ in the case that we have a nested sequence of piecewise linear function spaces.
\begin{prop}\label{lem:SpectralEquivance}
    Let $\Gamma$ be the surface of a Lipschitz polyhedron $\Omega$. Let $X_{\lambda,j} \subset X_{\lambda,j+1}$, $j=0,\dots, J-1$, denote a nested sequence of piecewise linear function spaces obtained by uniform dyadic refinements of the initial mesh (i.e., by structured refinements). Moreover, the functions~$\lambda_i $ (defined in \eqref{eq:pwlf}) are in $X_{\lambda,J}$, $J \in \N$, and $N_\mr{V} = \dim\(  X_{\lambda,J}\)$ is the number of vertices of the mesh.
    Then
    \begin{equation}
        \vec x^\T \hmat \Delta \vec x  \asymp \vec x^\T \hmat W \mat G_{\lambda \lambda}^{-1} \hmat W \vec x \quad\text{for all~} \vec x\in \R^{N_\mr{V}}  \, ,\label{eq:NewSpectralEquivLapW}
    \end{equation}
    holds uniformly in $J$ (and thus in $h$), where the matrices $\hmat \Delta$ and $\hmat W$ are defined in \eqref{eq:DeflectedLaplacian} and \eqref{eq:DiscreteModW}, respectively.
\end{prop}
\begin{proof}
    To prove this proposition, we leverage the stability results for the wavelet bases given in \cite{dahmen_elementbyelement_1999}. The bases presented therein are $H^s$-stable on globally Lipschitz continuous surfaces for $\abs{s} \leq 1$ \cite[p.~335]{dahmen_elementbyelement_1999}, which includes our case where we assumed that $\Gamma$ is the surface of a Lipschitz polyhedron.
    Let $\whmat \lambda \in \R^{N_\mr{V}\times N_\mr{V}}$ be the transformation matrix that maps the expansion coefficients of a function represented in terms of a wavelet basis $\widehat \lambda_{i}$ from \cite{dahmen_elementbyelement_1999} to the expansion coefficients of the same function represented in terms of piecewise linear basis $\lambda_i$.
    Then it follows from Theorem~ 2.1 in \cite{dahmen_elementbyelement_1999} (we note that in Theorem 2.1, the variable $\rho = 2$ for a nested sequence of piecewise linear functions as denoted in Remark~4.1 in \cite{dahmen_elementbyelement_1999}) that
    \begin{equation}
        \vec x^\T  \whmat \lambda^\T \hmat \Delta \whmat \lambda \vec x \asymp \vec x^\T \hmat D^{+2} \whmat \lambda^\T \mat G_{\lambda \lambda} \whmat \lambda  \hmat D^{+2} \vec x\quad\text{for all~} \vec x\in \R^{N_\mr{V}}\, ,  \label{eq:LaplaceHPree}
    \end{equation}
    and
    \begin{equation}
        \vec x^\T \whmat \lambda^\T \hmat W \whmat \lambda \vec x \asymp \vec x^\T \hmat D^{+1} \whmat \lambda^\T \mat G_{\lambda \lambda} \whmat \lambda  \hmat D^{+1} \vec x \quad\text{for all~} \vec x\in \R^{N_\mr{V}} \, , \label{eq:LaplaceWPree}
    \end{equation}
    where 
    \begin{equation}
        \matel{\hmat D}_{ii} = 2^{\hat l_\upLambda(i)/2}\, ,\quad i = 1, \dots, N_\mr{V}\, ,
    \end{equation}
    and the function $\hat l_\upLambda(i)$ returns the level on which the function $\widehat \lambda_{i}$ is defined. Equations \eqref{eq:LaplaceHPree} and \eqref{eq:LaplaceWPree} express the $H^1(\Gamma)$- and $H^{1/2}(\Gamma)$-stability, respectively, of the wavelet basis. It should be noted that the Sobolev spaces $\mathcal H^s(\Gamma)$ in Theorem 2.1 in \cite{dahmen_elementbyelement_1999} can be identified with “standard” $H^s(\Gamma)$ on globally Lipschitz continuous surfaces that are obtained via a partition of unity (see \cite[p.~335]{dahmen_elementbyelement_1999}). 
    
    If we assume that $\widehat \lambda_{i}$ are scaled\footnote{We note that if a realization of the wavelets is obtained following the description in \cite{dahmen_elementbyelement_1999}, then  $\norm{\widehat \lambda_{i}}_{L^2(\Gamma)} \asymp 1$ is, in general, not satisfied and thus the wavelets should be (trivially) rescaled to satisfy $\norm{\widehat \lambda_{i}}_{L^2(\Gamma)} \asymp 1$.} such that $\norm{\widehat \lambda_{i}}_{L^2(\Gamma)} \asymp 1$, then the $L^2(\Gamma)$-stability can be expressed as
    \begin{equation}
        \vec x^\T \vec x \asymp \vec x^\T \whmat \lambda^\T \mat G_{\lambda \lambda} \whmat \lambda  \vec x       \stackrel{\eqref{eq:Nodall2Stability}}{\asymp} \vec x^\T \whmat \lambda^\T \whmat \lambda  \vec x \, h^2 \quad\text{for all~} \vec x \in \R^{N_\mr{V}}\, . \label{eq:L2stability}
    \end{equation}
    Since all matrices in \eqref{eq:L2stability} are invertible, we (trivially) obtain
    \begin{equation}
        \vec x^\T  \vec x \asymp \vec x^\T \whmat \lambda^{-1}  \whmat \lambda^{-\T}  \vec x  / h^2 \quad\text{for all~} \vec x \in \R^{N_\mr{V}}\, . \label{eq:L2stabilityTrivial}
    \end{equation}
    Furthermore, we note that using \eqref{eq:L2stability} allows to simplify \eqref{eq:LaplaceHPree} and \eqref{eq:LaplaceWPree}  to 
    \begin{equation}
        \vec x^\T  \whmat \lambda^\T \hmat \Delta \whmat \lambda \vec x \asymp \vec x^\T \hmat D^4 \vec x \quad\text{for all~} \vec x\in \R^{N_\mr{V}}\, ,  \label{eq:LaplaceHPreeSimplfied}
    \end{equation}
    and
    \begin{equation}
        \vec x^\T \whmat \lambda^\T \hmat W \whmat \lambda \vec x \asymp \vec x^\T \hmat D^{2}\vec x  \quad\text{for all~} \vec x\in \R^{N_\mr{V}} \, . \label{eq:LaplaceWPreeSimplified}
    \end{equation}
    By using the substitution $\vec y  = \whmat \lambda  \vec x$ in \eqref{eq:LaplaceHPreeSimplfied} and \eqref{eq:LaplaceWPreeSimplified}, we obtain
    \begin{equation}
        \vec x^\T \hmat \Delta \vec x  \asymp  \vec x^\T \whmat \lambda^{-\T} \hmat D^{+4} \whmat \lambda^{-1}\vec x \quad\text{for all~} \vec x \in \R^{N_\mr{V}} \, , \label{eq:NewHowToPreconditionSqrtRootLap}
    \end{equation}
    and
    \begin{equation}
        \vec x^\T \hmat W \vec x \asymp  \vec x^\T \whmat \lambda^{-\T} \hmat D^{+2}  \whmat \lambda^{-1}\vec x  \quad\text{for all~} \vec x \in \R^{N_\mr{V}}\, . \label{eq:WSpectralEquiv}
    \end{equation}
    Summarizing, we obtain the spectral equivalence
    \begin{multline}
        \vec x^\T \hmat W \mat G_{\lambda \lambda}^{-1} \hmat W \vec x \stackrel{\eqref{eq:Nodall2Stability}}{\asymp}  \vec x^\T \hmat W^2 \vec x  / h^2 
        \stackrel{\eqref{eq:WSpectralEquiv}}{\asymp} \vec x^\T \(\whmat \lambda^{-\T} \hmat D^{+2}  \whmat \lambda^{-1} \)^2 \vec x / h^2 \\ \asymp \vec x^\T \whmat \lambda^{-\T} \hmat D^{+2}  \whmat \lambda^{-1} \whmat \lambda^{-\T} \hmat D^{+2}  \whmat \lambda^{-1}  \vec x / h^2  
        \stackrel{\eqref{eq:L2stabilityTrivial}}{\asymp}\vec x^\T \whmat \lambda^{-\T} \hmat D^{+4}  \whmat \lambda^{-1}  \vec x    
        \stackrel{\eqref{eq:NewHowToPreconditionSqrtRootLap}}{\asymp}   \vec x^\T \hmat \Delta \vec x \quad\text{for all~} \vec x \in \R^{N_\mr{V}}\, .
    \end{multline}
\end{proof}
    
    \bibliographystyle{elsarticle-num}
    \bibliography{Literature}

\end{document}